\documentclass[11pt]{amsart} \usepackage{amssymb} \usepackage{amsmath}
\usepackage{amsthm} \usepackage{amscd}
\title[On multiplier systems and theta functions]
{On multiplier systems and theta functions of half-integral weight 
for the Hilbert modular group $\mathrm{SL}_2(\mathfrak{o})$} 
\author{Hiroshi Noguchi} 
\begin{document} 
\maketitle \theoremstyle{definition} 
\newtheorem{dfn}{Definition} \newtheorem{thm}{Theorem}
\newtheorem{prop}{Proposition} \newtheorem{cor}{Corollary}
\newtheorem{lem}{Lemma}
\newcommand{\mz}{\mathbb{Z}} \newcommand{\mq}{\mathbb{Q}} 
\newcommand{\ma}{\mathbb{A}} \newcommand{\mr}{\mathbb{R}}
\newcommand{\mac}{\mathbb{C}} \newcommand{\wt}[1]{\widetilde{\mathrm{SL}_2(#1)}} \newcommand{\ms}{\mathfrak{S}} \newcommand{\maa}{\mathfrak{a}}
\newcommand{\mo}{\mathfrak{o}} \newcommand{\map}{\mathfrak{p}} 
\newcommand{\mb}{\mathfrak{b}} \newcommand{\mh}{\mathfrak{h}} 
\newcommand{\mk}{\mathbf{k}} \newcommand{\me}{\mathbf{e}}
\newcommand{\md}{\mathfrak{d}} \newcommand{\mv}{\mathbf{v}}

\begin{abstract} 
Let $F$ be a totally real number field and $\mo$ the ring of integers of $F$. 
We study theta functions which are Hilbert modular forms of half-integral weight 
for the Hilbert modular group $\mathrm{SL}_2(\mathfrak{o})$. 
We obtain an equivalent condition that there exists a multiplier system of half-integral weight 
for $\mathrm{SL}_2(\mo)$. 
We determine the condition of $F$ that there exists a theta function 
which is a Hilbert modular form of half-integral weight for $\mathrm{SL}_2(\mathfrak{o})$. 
The theta function is defined by a sum on a fractional ideal $\maa$ of $F$. 
\end{abstract}

\section{Introduction} \label{intro} 
Put $e(z)=e^{2\pi i z}$ for $z \in \mac$. 
It is known that the modular forms of $\mathrm{SL}_2(\mz)$ of weight 1/2 and 3/2 are 
the Dedekind eta function $\eta(z)$ and its cubic power $\eta^3(z)$ up to constant, respectively. 
Here, $\eta(z)$ is given by 
$$\eta(z)=e(z/24)\prod_{m\ge 1} (1-e(mz)) \hspace{15pt} (z \in \mathfrak{h}), $$ 
where $\mh$ is the upper half plane. 
It is known that 
$$\eta(z)=\frac{1}{2} \sum_{m \in \mz} \chi_{12}(m) e(mz/24), \hspace{15pt} 
\eta^3(z)=\frac{1}{2} \sum_{m \in \mz} m\, \chi_{4}(m) e(mz/8). $$ 
Here, $\chi_{12}$ and $\chi_4$ are the primitive character mod~12 and mod~4, respectively. 
Note that $\eta(z)$ and $\eta^3(z)$ are theta functions defined by a sum on $\mz$. 

The function $\eta(z)$ has the transformation formula with respect to modular transformations 
(see \cite{Pe2, Ra1, We1}). 
Let $\left(\dfrac{\cdot}{\cdot}\right)$ be the Jacobi symbol. 
We define $\left(\dfrac{\cdot}{\cdot}\right)^*$ and 
$\left(\dfrac{\cdot}{\cdot}\right)_*$ by 
$$\left(\frac{c}{d}\right)^*=\left(\frac{c}{|d|}\right), \hspace{15pt} 
\left(\frac{c}{d}\right)_*=t(c, d) \left(\frac{c}{d}\right)^*, \hspace{15pt} 
t(c, d)=\begin{cases}-1 & c, d<0 \\ 1 & \text{otherwise, } \end{cases} $$ 
for $c \in \mz \backslash \{0\}$ and $d \in 2\mz+1$ such that $(c, d)=1$. 
We understand 
$$\left(\frac{0}{\pm 1}\right)^*=\left(\frac{0}{1}\right)_*=1, \hspace{15pt} 
\left(\frac{0}{-1}\right)_*=-1$$ 
(see \cite[Chapter 4 \S 1]{Kn1}). 

For $g \in \mathrm{SL}_2(\mr)$ and $z \in \mh$, put 
\begin{equation} \label{capJ} 
J(g, z)=\begin{cases}\sqrt{d} & \text{if } c=0, d>0 \\ 
-\sqrt{d} & \text{if } c=0, d<0 \\ (cz+d)^{1/2} & \text{if } c \ne 0, \end{cases} 
\hspace{15pt} g=\begin{pmatrix} a & b \\ c& d \end{pmatrix}. 
\end{equation} 
Here, we choose arg$(cz+d)$ such that $-\pi<$arg$(cz+d) \le \pi$. 
Then we have 
\begin{equation} \label{form} \eta(\gamma(z))=\mv_\eta(\gamma)J(\gamma, z)\eta(z), \hspace{15pt} 
\gamma(z)=\dfrac{az+b}{cz+d} \in \mh
\end{equation} 
for any $\gamma=\begin{pmatrix} a & b \\ c& d \end{pmatrix} \in \mathrm{SL}_2(\mz)$, 
where the multiplier system $\mv_{\eta}(\gamma)$ is given by  
\begin{equation} \label{etamul} 
\mv_{\eta}(\gamma)=\begin{cases} \left(\dfrac{d}{c}\right)^* e\left(\dfrac{(a+d)c-bd(c^2-1)-3c}{24}\right)&c \text{ : odd } \\ \noalign{\vskip 10pt}
\left(\dfrac{c}{d}\right)_* e\left(\dfrac{(a+d)c-bd(c^2-1)+3d-3-3cd}{24}\right) &c\text{ : even.} 
 \end{cases} \end{equation}

It is natural to ask the following problem. 
When does a Hilbert modular theta series of weight 1/2 with respect to $\mathrm{SL}_2(\mo)$ exist? 
Here, $\mo$ is the ring of integers of a totally real number field $F$. 
In 1983, Feng \cite{Fe1} studied this problem.   
She gave a sufficient condition for the existence of a Hilbert modular theta series of weight 1/2 
with respect to $\mathrm{SL}_2(\mo)$ and constructed certain Hilbert modular theta series. 
These series are defined by a sum on $\mo$. 

Let $K$ be a real quadratic field and $d_K$ the discriminant of $K$. 
Gundlach \cite[p.30]{Gu1}, \cite[Remark 4.1.]{Gu2} showed that 
if $d_K \equiv 1$ mod~8, then there exist multiplier systems of weight 1/2 
for a Hilbert modular group belonging to a certain theta series.
Naganuma \cite{Na1} obtained a Hilbert modular form of level 1 for a real quadratic $\mathbb{Q}(\sqrt{D})$, \ $D\equiv 1$ mod $8$ with class number one, using modular imbeddings, 
from the theta constant with the characteristic (1/2, 1/2, 1/2, 1/2) of degree 2. 

In this paper, we solve the problem above completely. 
We consider theta functions defined by a sum on a fractional ideal $\maa$ of $F$. 
Let $v$ be a place of $F$ and $F_v$ the completion of $F$ at $v$. 
When $v$ is a finite place, we write $v<\infty$. 
When $v$ is an infinite place, we have $F_v \simeq \mr$ and write $v \mid \infty$. 
Let $\ma$ be the adele ring of $F$. 

Let $n=[F:\mq]$ and $\iota_v:F \to F_v$ be the embedding for any $v$. 
The entrywise embeddings of $\mathrm{SL}_2(F)$ into $\mathrm{SL}_2(F_v)$ 
are also denoted by $\iota_v$. 
The metaplectic group of $\mathrm{SL}_2(F_v)$ is denoted by $\wt{F_v}$, 
which is a nontrivial double covering group of $\mathrm{SL}_2(F_v)$. 
Set-theoretically, it is $\{[g, \tau] \mid g \in \mathrm{SL}_2(F_v), \tau \in \{\pm 1\}\}$. 
Its multiplication law is given by $[g, \tau][h, \sigma]=[gh, \tau\sigma c(g, h)]$ 
for $[g, \tau], [h, \sigma] \in \wt{F_v}$, 
where $c(g, h)$ is the Kubota 2-cocycle on $\mathrm{SL}_2(F_v)$. 
Put $[g]=[g, 1]$. 

Let $\{\infty_1, \cdots, \infty_n\}$ be the set of infinite places of $F$. 
Put $\iota_i=\iota_{\infty_i}$ for $1 \le i \le n$. 
We embed $\mathrm{SL}_2(F)$ into $\mathrm{SL}_2(\mr)^n$ by $r \mapsto (\iota_1(r), \cdots, \iota_n(r))$. 
We denote the embedding of $\mathrm{SL}_2(F)$ into $\mathrm{SL}_2(\ma)$ by $\iota$. 
Let $\ma_f$ be the finite part of $\ma$ 
and $\iota_f:$ $\mathrm{SL}_2(F) \to$ $\mathrm{SL}_2(\ma_f)$ the projection of the finite part. 
The embedding of $F$ into $\ma_f$ is also denoted by $\iota_f$. 

Let $\wt{\ma}$ be the adelic metaplectic group, which is a double covering of 
$\mathrm{SL}_2(\ma)$. 
Let $\tilde{H}$ be the inverse image of a subgroup $H$ of $\mathrm{SL}_2(\ma)$ in $\wt{\ma}$. 
It is known that $\mathrm{SL}_2(F)$ can be canonically embedded into $\wt{\ma}$.
The embedding $\tilde{\iota}$ is given by $g \mapsto ([\iota_v(g)])_v$ 
for each $g \in$ $\mathrm{SL}_2(F)$. 
We define the maps $\tilde{\iota}_f:\mathrm{SL}_2(F) \to \wt{\ma_f}$ 
and $\tilde{\iota}_{\infty}:\mathrm{SL}_2(F) \to \wt{F_{\infty}}$ by 
\[
\tilde{\iota}_f(g)=([\iota_v(g)])_{v<\infty} \times ([1_2])_{v \mid \infty}, \hspace{15pt} 
\tilde{\iota}_{\infty}(g)=([1_2])_{v<\infty} \times ([\iota_i(g)])_{v \mid \infty}, 
\]
where $1_2$ is the unit matrix of size 2. 
Then we have $\tilde{\iota}(g)=\tilde{\iota}_f(g) \tilde{\iota}_{\infty}(g)$ for any 
$g \in \mathrm{SL}_2(F)$. 

Let $\Gamma \subset$ $\mathrm{SL}_2(\mo)$ be a congruence subgroup. 
A map $\mv:\Gamma \to \mac^{\times}$ is said to be a multiplier system of half-integral weight 
if $\mv(\gamma) \prod_{i=1}^n J(\iota_i(\gamma), z_i)$ is an automorphy factor 
for $\Gamma \times \mh^n$, where $J$ is the function in (\ref{capJ}). 
We obtain an equivalent condition that $\mv$ is a multiplier system of half-integral weight. 
Let $K_{\Gamma}$ be the closure of $\iota_f(\Gamma)$ in $\mathrm{SL}_2(\ma_f)$
and $\tilde{K}_{\Gamma}$ the inverse image of $K_{\Gamma}$ in $\wt{\ma_f}$. 
Let $\lambda:\tilde{K}_{\Gamma} \to \mac^{\times}$ be a genuine character. 
Put $\mv_{\lambda}(\gamma)=\lambda(\tilde{\iota}_f(\gamma))$ 
for $\gamma \in \Gamma$. 
Then $\mv_{\lambda}$ is a multiplier system of half-integral weight for $\Gamma$. 

Now suppose that $\mv:\Gamma \to \mac^{\times}$ is a multiplier system of half-integral weight. 
We obtain an equivalent condition that there exists a genuine character $\lambda:\tilde{K}_{\Gamma} \to \mac^{\times}$ such that $\mv_{\lambda}=\mv$. 
Put $K_f=\prod_{v<\infty} \mathrm{SL}_2(\mo_v)$, 
which $\mo_v$ is the ring of integers of $F_v$. \vspace{5pt}

\noindent {\bf Proposition 2.}
If $F \ne \mq$, then any multiplier system $\mv$ of half-integral weight 
of any congruence subgroup $\Gamma \subset \mathrm{SL}_2(\mo)$ is obtained 
from a genuine character of $\tilde{K}_{\Gamma}$. \vspace{5pt}

\noindent {\bf Proposition 3.}
Let $\mv$ be a multiplier system of half-integral weight for $\mathrm{SL}_2(\mo)$. 
Then there exists a genuine character 
$\lambda:\tilde{K}_f \to \mac^{\times}$ such that $\mv_{\lambda}=\mv$. \vspace{5pt}

\noindent {\bf Corollary 2.}
There exists a multiplier system $\mv$ of half-integral weight for $\mathrm{SL}_2(\mo)$ 
if and only if 2 splits completely in $F/\mq$. 
There exists a genuine character of $\wt{\mo_v}$ for any $v<\infty$, 
provided that this condition holds. \vspace{5pt}

Let $\psi: \ma/F \to \mac^{\times}$ be an additive character such that its $v$-component 
$\psi_{v}(x)$ equals $e(x)$ for any $v \mid \infty$. 
Put $\psi_{\beta}(x)=\psi(\beta x)$ and $\psi_{\beta, v}(x)=\psi_v(\beta x)$ 
for $\beta \in F^{\times}$. 
The Schwartz space of $F_v$ is denoted by $S(F_v)$. 
Let $\omega_{\psi_{\beta}, v}$ be the Weil representation of the metaplectic group $\wt{F_v}$ 
on $S(F_v)$ corresponding to $\psi_{\beta, v}$. 

In the case $v<\infty$, 
we shall determine the genuine characters of the metaplectic group $\wt{\mo_v}$. 
Let $\lambda_v$ be a genuine character of $\wt{\mo_v}$. 
The space $(\omega_{\psi_{\beta}, v}, S(F_v))^{\lambda_v}$ is defined by a set of 
$f \in S(F_v)$ such that $\omega_{\psi_{\beta}, v}(g)f=\lambda_v(g)^{-1}f$ for any $g \in \wt{\mo_v}$. 
We determine the space completely. 

In the case $v \mid \infty$, let $\lambda_v$ be a genuine character of the metaplectic group 
$\widetilde{\mathrm{SO(2)}}$, where $\mathrm{SO}(2)$ is a set of 
$\begin{pmatrix} a & b \\ -b & a \end{pmatrix} \in$ $\mathrm{SL}_2(\mr)$. 
The space $(\omega_{\psi_{\beta}, v}, S(\mr))^{\lambda_v}$ is defined by a set of 
$f \in S(\mr)$ such that $\omega_{\psi_{\beta}, v}(g)f=\lambda_v(g)f$ 
for any $g \in \widetilde{\mathrm{SO(2)}}$. 
We have an irreducible decomposition 
\[
\omega_{\psi_{\beta}, v}=\omega_{\psi_{\beta}, v}^+ \oplus \omega_{\psi_{\beta}, v}^-,
\]
where $\omega_{\psi_{\beta}, v}^+$ (resp. $\omega_{\psi_{\beta}, v}^-$) is 
an irreducible representation of the set of even (resp. odd) functions in $S(\mr)$ 
 (see \cite[Lemma 2.4.4]{LiVe1}). 
 
If $\beta<0$, there exist no lowest weight vectors of $\omega_{\psi_{\beta}, v}^+$ or 
$\omega_{\psi_{\beta}, v}^-$. 
If $\beta>0$, 
the vector $e(i \iota_v(\beta) x^2)$ (resp. $x e(i \iota_v(\beta) x^2)$) is the lowest weight vector 
of $\omega_{\psi_{\beta}, v}^+$ (resp. $\omega_{\psi_{\beta}, v}^-$) of weight 1/2 (resp. 3/2) 
(see \cite[Lemma 2.4.4]{LiVe1}). 
Let  $\lambda_{\infty, 1/2}$  be a genuine character of lowest weight 1/2 
with respect to $(\omega_{\psi_{\beta}, v}^+, S(\mr))$ and $\lambda_{\infty, 3/2}$ 
of lowest weight 3/2 with respect to $(\omega_{\psi_{\beta}, v}^-, S(\mr))$. 

Now suppose that 2 splits completely in $F/\mq$. 
The set of totally positive elements of $F$ is denoted by $F_+^{\times}$. 
Assume that $\beta \in F_+^{\times}$ in order that there exists a lowest weight vector of 
$(\omega_{\psi_{\beta}, v}^+, S(\mr))$ or  $(\omega_{\psi_{\beta}, v}^-, S(\mr))$ 
for any $v<\infty$. 
We fix $\omega_{\psi_{\beta}, v}$ and $\lambda_v$ for any $v$. 
Here, we assume that $\lambda_v=\lambda_{\infty, 1/2}$ 
or $\lambda_v=\lambda_{\infty, 3/2}$ for any $v \mid \infty$. 
Put $K=K_f \times \prod_{v \mid \infty} \mathrm{SO}(2)$. 
Let $\lambda:\tilde{K} \to \mac^{\times}$ be a genuine character such that 
its $v$-component is $\lambda_v$.
Let $S(\ma)$ be the Schwartz space of $\ma$. 
The space $(\omega_{\psi_{\beta}}, S(\ma))^{\lambda}$ is defined by a set of $\phi=\prod_v \phi_v 
\in S(\ma)$ such that $\phi_v \in (\omega_{\psi_{\beta}, v}, S(F_v))^{\lambda_v}$ for any $v$. 
We determine when there exists a nonzero $\phi \in (\omega_{\psi_{\beta}}, S(\ma))^{\lambda}$. 
We define the theta function $\Theta_{\phi}$ by 
$$\Theta_{\phi}(g)=\sum_{\xi \in F} \omega_{\psi_{\beta}}(g)\phi(\xi)$$ 
for $\phi \in S(\ma)$ and $g \in \wt{\ma}$, 
where $\omega_{\psi_{\beta}}(g)\phi(\xi)=\prod_v \omega_{\psi_{\beta}, v}(g_v)\phi_v(\xi)$. 
The product is essentially a finite product. 
If $\phi \in (\omega_{\psi_{\beta}}, S(\ma))^{\lambda}$ for $\lambda$ 
such that the lowest weight of $(\omega_{\psi_{\beta}}, S(\ma))^{\lambda}$ is 1/2 (resp. 3/2), 
then it is known that $\Theta_{\phi}$ is a Hilbert modular form of weight 1/2 (resp. 3/2). 

For $1 \le i \le n$, put $\lambda_{\infty_i}=\lambda_{\infty, w_i}$, where $w_i=1/2$ or $3/2$. 
Put $S_{\infty}=\{\infty_i \mid w_i=3/2 \}$ and $S_2=\{v<\infty \mid F_v=\mq_2\}$. 
Let $\map_v$ be the maximal ideal of $\mo_v$ and $q_v$ the order of $\mo_v/\map_v$.
Put $T_3=\{v<\infty \mid q_v=3\}$. 
We denote the order of a set $S$ by $|S|$. 
Let $\mathbf{G}$ be the set of triplets $(\beta, S_3, \maa)$ of $\beta \in F_+^{\times}$, 
a subset $S_3 \subset T_3$ and a fractional ideal $\maa$ of $F$ satisfying the conditions 
\[
|S_2|+|S_3|+|S_{\infty}| \in 2\mz 
\]
and
\[
(8\beta) \md \prod_{v \in S_3} \map_v=\maa^2,  
\]
where $\md$ is the different of $F/\mq$. 
We define an equivalence relation $\sim$ on $\mathbf{G}$ by 
\[
(\beta, S_3, \maa) \sim (\beta', S'_3, \maa') \Longleftrightarrow 
S_3=S'_3, \ \beta'=\gamma^2 \beta, \ \maa'=\gamma \maa 
\text{ for some } \gamma \in F^{\times}. 
\]

We determine when there exists a nonzero $\Theta_{\phi}$. 
Recall that if $q_v$ is odd, the double covering $\wt{F_v} \to$ $\mathrm{SL}_2(F_v)$ splits on 
$\mathrm{SL}_2(\mo_v)$. 
We denote the image of $g \in \mathrm{SL}_2(\mo_v)$ under the splitting by $[g, s(g)]$. 
Thus if $q_v$ is odd, there exists a genuine character $\epsilon_v:\wt{\mo_v} \to \mac^{\times}$ 
satisfying $\epsilon_v([g, s(g)])=1$ for any $g \in \mathrm{SL}_2(F_v)$. 
Now we set $S_3=\{v \mid q_v=3, \lambda_v \ne \epsilon_v \}$. \vspace{5pt}

\noindent {\bf Theorem 1.}
Suppose that 2 splits completely in $F/\mq$. 
Let $\beta \in F^\times_+$, $\lambda:\tilde{K} \to \mac^{\times}$ 
and $w_1, \ldots, w_n\in\{1/2, 3/2\}$ be as above.
Then there exists $\phi=\prod_v \phi_v \in (\omega_{\psi_{\beta}}, S(\ma))^{\lambda}$ 
such that $\Theta_\phi\neq 0$ if and only if 
there exists a fractional ideal $\maa$ of $F$ such that $(\beta, S_3, \maa) \in \mathbf{G}$. 
\vspace{5pt}

Put 
\[
H=\left\{\prod_{\begin{smallmatrix} {v\in T_3}\end{smallmatrix}} \map_v^{e_v} \mid 
\sum_{v \in T_3} e_v \in 2\mz \right\}. 
\]
Let $\mathrm{Cl}^+$ be the narrow ideal class group of $F$.
Put $\mathrm{Cl}^{+2}=\{\mathfrak{c}^2 \mid \mathfrak{c} \in \mathrm{Cl}^+ \}$. 
We denote the image of the group $H$ (resp. $\mathfrak{b} \in \mathrm{Cl}^+$) 
in $\mathrm{Cl}^+/\mathrm{Cl}^{+2}$ by $\bar{H}$ (resp. $[\mathfrak{b}]$). \vspace{5pt}

\noindent {\bf Theorem 2.} 
Suppose that 2 splits completely in $F/\mq$. Let $w_1, \ldots, w_n\in\{1/2, 3/2\}$ be as above.
\begin{itemize}
\item[(1)] 
Suppose that $|S_2|+|S_{\infty}|$ is even.
Then there exists $(\beta, S_3, \maa) \in \mathbf{G}$ if and only if $[\md]\in\bar{H}$. 
\item[(2)] 
Suppose that $|S_2|+|S_{\infty}|$ is odd.
Then there exists $(\beta, S_3, \maa) \in \mathbf{G}$ if and only if $T_3\neq \emptyset$ 
and $[\md \map_{v_0}]\in \bar{H}$. Here, $v_0$ is any fixed element of $T_3$.
\end{itemize} 

Now suppose that there exists $(\beta, S_3, \maa) \in \mathbf{G}$. 
Replacing $\beta$ with $\beta \gamma^2$ and $\maa$ with $(\maa \gamma)^2$ 
in (\ref{class}), respectively, we may assume $\mathrm{ord} _v \maa=0$ for $v \in S_2 \cup S_3$. 
For $v \in S_2 \cup S_3$, put 
\[
f_v(x)=\begin{cases} 1 & \text{ if } x \in 1+2\map_v \\ -1 & \text{ if } x \in -1+2\map_v \\ 
0 & \text{ otherwise.} \end{cases}
\] 
We set
\[
f=\prod_{v \in S_2 \cup S_3} f_v \times 
\prod_{v<\infty, v \notin S_2 \cup S_3} \mathrm{ch }\maa_v^{-1},
\] 
where $\maa_v=\maa \mo_v$. 
Put $\phi=f \times \prod_{i=1}^n f_{\infty, i}$, 
where $f_{\infty, i}(x)=x^{w_i-(1/2)} e(i\iota_i(\beta) x^2)$ for $x \in \mr$ and $w_i \in \{1/2, 3/2\}$. 
By Theorem \ref{main1}, there exists $\Theta_{\phi} \ne 0$ of weight $w=(w_1, \cdots, w_n)$. 
 \vspace{5pt}

\noindent {\bf Theorem 3.} 
Let $\phi$ and $\Theta_{\phi}$ be as above. 
We define a theta function $\theta_{\phi}: \mh^n \to \mac$ by 
\[
\theta_{\phi}(z)=\sum_{\xi \in \maa^{-1}} f(\iota_f(\xi)) 
\prod_{\infty_i \in S_{\infty}} \iota_i(\xi) \prod_{i=1}^n e(z_i \iota_i(\beta \xi^2)). 
\] 
for $z=(z_1, \cdots, z_n) \in \mh$. 
Then $\theta_{\phi}$ is a nonzero Hilbert modular form of weight $w$ 
for $\mathrm{SL}_2(\mo)$ with respect to a multiplier system. 
Every theta function of weight $w$ for $\mathrm{SL}_2(\mo)$ with a multiplier system 
may be obtained in this way. 

In particular, when $F=\mq$, we obtain $\eta(z)$ and $\eta^3(z)$ 
as $\theta_{\phi}(z)$ up to constant. 

This paper is organized as follows. 
In Section \ref{local}, we determine the number of the genuine characters of the metaplectic 
group $\wt{\mo}$, 
where $\mo$ is the ring of integers of a finite extension $F$ of $\mq_p$. 
Moreover, we determine the dimension of a space $(\omega_{\psi_{\beta}}, S(F))^{\lambda}$ 
for a genuine character $\lambda$ of $\wt{\mo}$ and the Weil 
representation $\omega_{\psi_{\beta}}$ of $\wt{F}$ on $S(F)$. 
In Section \ref{system}, 
we study the multiplier systems of half-integral weight of a congruence subgroup of 
$\mathrm{SL}_2(\mo)$, where $\mo$ is the ring of integers of a totally real number field $F$. 
In Section \ref{main}, we define theta functions $\Theta_{\phi}$ of $\wt{\ma}$ 
and prove our main theorems. 
Moreover, we obtain theta functions $\theta_{\phi}(z)$ of $\mh^n$ and 
determine the number of the equivalence classes of the set $\mathbf{G}$. 
In Section \ref{quad}, we give some examples in the case $F=\mq$ or $F$ is a real quadratic field. 
\vspace{5pt}

\noindent {\bf Acknowledgment.}
The author thanks his supervisor Tamotsu Ikeda for suggesting the problem and for his helpful
advice, and thanks Masao Oi and Shuji Horinaga for their sincere and useful comments.

\section{Local theory, genuine characters of $\wt{\mo}$} \label{local} 
Let $F$ be a finite extension of $\mq_p$ until the end of this section. 
Let $\mo$ be the ring of integers of $F$ and $\map$ the maximal ideal of $\mo$. 
Let $q$ be the order of the residue field $\mo/\map$ and $\md$ the different of $F/\mq_p$. 

For $g=\begin{pmatrix}a & b \\ c & d \end{pmatrix} \in \mathrm{SL}_2(F)$, 
put $x(g)=c$ if $c \ne 0$ and $x(g)=d$ if $c=0$. The Kubota 2-cocycle on 
$\mathrm{SL}_2(F)$ is defined by $c(g, h)=\langle x(g)x(gh), x(h)x(gh) \rangle_F$ 
for $g, h \in \mathrm{SL}_2(F)$, where $\langle \cdot, \cdot \rangle_F$ 
is the quadratic Hilbert symbol for $F$. 
Let $\wt{F}$ be the metaplectic group of $\mathrm{SL}_2(F)$. Set-theoretically, it is 
\[
\{[g, \tau] \mid g \in \mathrm{SL}_2(F), \tau \in \{\pm 1\} \}. 
\]
Its multiplication law is given by $[g, \tau][h, \sigma] = [gh, \tau \sigma c(g, h)]$. 
This is a nontrivial double covering group of $\mathrm{SL}_2(F)$. 
Put $[g]=[g, 1]$.
For a subgroup $H$ of $\mathrm{SL}_2(F)$, 
the inverse image of $H$ in $\wt{F}$ is denoted by $\tilde{H}$. 
A function $\epsilon_F:\wt{\mo} \to \mac$ is genuine 
if $\epsilon_F([1_2, -1]\gamma)=-\epsilon_F(\gamma)$ for any $\gamma \in \wt{\mo}$. 

We determine the number of the genuine characters of $\wt{\mo}$. 
For $a \in F^{\times}, \tau=\pm1$ and $b, c \in F$, put 
\[
m(a, \tau)=\left[\begin{pmatrix} a & 0 \\ 0 & a^{-1} \end{pmatrix}, \tau \right], \hspace{15pt} 
u^+(b)=\left[\begin{pmatrix} 1 & b \\ 0 & 1 \end{pmatrix}\right], 
\]
\[ 
u^-(c)=\left[\begin{pmatrix}1 & 0 \\ c & 1\end{pmatrix}\right], \hspace{15pt} 
N=\left[\begin{pmatrix} 0 & -1 \\ 1 & 0 \end{pmatrix}\right]. 
\]

For $k \in \mz$ such that $k \ge 0$, we define the subgroups $U^+(\map^k)$, $U^-(\map^k)$ 
and $\tilde{A}$ of $\wt{\mo}$ by  
$U^+(\map^k)=\{u^+(b) \mid b \in \map^k\}$, $U^-(\map^k)=\{u^-(c) \mid c \in \map^k\}$ 
and $\tilde{A}=\{m(a, \tau) \mid a \in \mo^{\times} \}$, respectively. 
Note that $\wt{\mo}$ is generated by $U^+(\mo)$, $N$ and $m(1, -1)$. 

\begin{lem} \label{order}
Put $M=$ min$\{\mathrm{ord}(a^2-1) \mid a \in \mo^{\times}\}$. Then we have 
$$M= \begin{cases} 0 & (q \ge 4) \\ 1 & (q=3) \\ 
2 & (q=2, F \ne \mq_2) \\ 3 & (F=\mq_2). \end{cases}$$ \end{lem}

\begin{proof} 
Let $\pi$ be a prime element of $F$. 
If $q \ge 4$, then there exists $a \in \mo^{\times}$ such that $a^2-1 \in \mo^{\times}$. 
Thus we have $M=0$. 
If $q=3$, then $a^2-1 \in \map$ for any $a \in \mo^{\times}$. 
Since $(\pi+1)^2-1=\pi(\pi+2) \notin \map^2$, we have $M=1$. 

In the case $q=2$,  $a^2-1=(a-1)(a+1) \in \map^2$ for any $a \in \mo^{\times}$. 
If $F \ne \mq_2$, then we have 2 $\in \map^2$. 
Since $(\pi+1)^2-1=\pi(\pi+2) \notin \map^3$, we have $M=2$. 
It is well-known that $M=3$ if $F=\mq_2$.  \end{proof}

The derived group of a group $G$ is denoted by $D(G)$. 
Since $[m(a, \tau), u^+(b)]=u^+((a^2-1)b)$ and $[m(a, \tau), u^-(c)]=u^-((a^{-2}-1)c)$ hold, 
we have 
\begin{equation} \label{sub} U^+(\map^M), U^-(\map^M) \subset D(\wt{\mo})\end{equation} 
by Lemma \ref{order}. 

\begin{lem} \label{Hsym} 
Suppose that $q$ is even. Then there exists $r \in 1+4 \mo$ such that 
$\langle r, x \rangle_F=(-1)^{\mathrm{ord }x}$ for $x \in F^{\times}$. \end{lem} 

\begin{proof} 
Put $r=1+4c$ for $c \in \mo$. 
We show that there exists $c$ such that $F(\sqrt{r})/F$ is an unramified quadratic extension. 
We denote the residue field of a local field $L$ by $k(L)$ 
and the image of an element $u$ of the ring of integers of $F(\sqrt{r})$ 
in $k(F(\sqrt{r}))$ by $\bar{u}$. 

We define a map $\mathbf{p}:k(F) \to k(F)$ by $\mathbf{p}(t)=t^2-t$ for $t \in k(F)$. 
We have $\mathbf{p}(t)=\mathbf{p}(1-t) \ne \mathbf{p}(s)$ 
for any $s \in k(F) \backslash \{t, 1-t\}$. 
Since $[k(F):\mathbf{p}(k(F))]=2$, there exists $c$ such that $\bar{c} \notin \mathbf{p}(k(F))$. 
Then it is known that a polynomial $X^2-X-\bar{c}$ is irreducible over $k(F)$. 
Put $y=(1-\sqrt{r})/2$ and $f(X)=X^2-X-c \in \mo[X]$. 
Since $f(y)=0$ and $f'(y)=2y-1=-\sqrt{r} \ne 0$, $k(F)(\bar{y})/k(F)$ is a quadratic extension 
and $k(F(\sqrt{r}))=k(F(y))$ equals $k(F)(\bar{y})$. 
Therefore $F(\sqrt{r})/F$ is an unramified quadratic extension. 
\end{proof} 

\begin{lem} \label{noexi} 
Suppose that $q$ is even and $F \ne \mq_2$. 
Then there exist no genuine characters of $\wt{\mo}$. 
\end{lem}

\begin{proof} 
Let $b, c \in \mo$ such that $r=1-bc \in \mo^{\times}$ and $\zeta=\langle r, b \rangle_F$. 
We have 
\begin{equation} \label{comm} 
[u^-(c), u^+(b)]=\left[\begin{pmatrix}r & b^2 c \\ -bc^2 & 1+bc+b^2 c^2 \end{pmatrix}\right]
=u^-(-r^{-1}bc^2)m(r, \zeta)u^+(r^{-1}b^2c). \end{equation} 

When $F/\mq_2$ is a ramified extension, we assume that $b, c \in 2\mo$ 
such that $r$ satisfies the condition in Lemma \ref{Hsym}. 
We have $U^+(\map^2), U^-(\map^2) \subset D(\wt{\mo})$ by (\ref{sub}). 
Then we have $m(r, \zeta) \in D(\wt{\mo})$ by (\ref{comm}). 
Let $\pi$ be a prime element of $F$. 
Set $b'=b\pi$ and $c'=c\pi^{-1}$, which lie in $\map$. 
We have $\langle 1-b'c', b' \rangle_F=\langle r, b\pi \rangle_F=-\zeta$. 
Thus we have $m(1, -1) \in D(\wt{\mo})$ and there exist no genuine characters of $\wt{\mo}$. 

Assume that $F/\mq_2$ is an unramified extension and that $F \ne \mq_2$. 
We have $U^+(\mo)$, $U^-(\mo) \subset D(\wt{\mo})$ by (\ref{sub}). 
Substituting $1$ for $c$ in (\ref{comm}), we have $m(1-b, \zeta)$ $\in D(\wt{\mo})$, 
whenever $1-b \in \mo^{\times}$. 
Since $\zeta=\langle 1-b, b \rangle_F$ equals 1, we have $m(1-b, 1) \in D(\wt{\mo})$. 
Similarly, substituting $-1$ for $c$ and replacing $b$ with $-b$ in Equation (\ref{comm}), 
we have $m(1-b, \langle 1-b, -1 \rangle_F) \in D(\wt{\mo})$. 
Thus it suffices to show that there exists $b$ such that $\langle 1-b, -1 \rangle_F$ equals $-1$. 

Since $F/\mq_2$ is unramified, $F(\sqrt{-1})/F$ is a ramified extension. 
Thus there exists $u \in \mo^{\times}$ such that $\langle u, -1 \rangle_F=-1$. 
Since $[\mo^{\times}:1+\map]=q-1$ is odd, we may assume that $u \in 1+\map$. 
Then there exists $b \in \map$ such that $u=1-b$ satisfies $\langle u, -1 \rangle_F=-1$. 
\end{proof} 

An additive character $\me_p$ of $\mq_p$ is defined by 
$\me_p(x)=e(-x)$ for any $x \in \mz[1/p]$. 
We define a nontrivial additive character $\psi_{\beta}$ of $F$ 
by $x \mapsto \me_p($Tr$_{F/\mq_p}(\beta x))$ for $\beta \in F^{\times}$. 
The order of $\psi_{\beta}$ is denoted by $\mathrm{ord} \psi_{\beta} \in \mz$, 
which is defined by $\psi_{\beta}(\map^{-\mathrm{ord}\psi_{\beta}})=1$ 
and $\psi_{\beta}(\map^{-\mathrm{ord}\psi_{\beta} -1}) \ne 1$. 
We have $\mathrm{ord} \psi_{\beta}=\mathrm{ord} \md+ \mathrm{ord} \beta$.

Let $S(F)$ be the Schwartz space of $F$. 
The Fourier transformation $\hat{\phi}$ of $\phi \in S(F)$ is defined by 
$\hat{\phi}(x)=\int_{F} \phi(y) \psi_{\beta}(xy)dy$. 
Here, $dy$ is self-dual on the Fourier transformation. 
In other words, $dy$ is the Haar measure such that the Plancherel's
formula $\int_{F} |\phi(y)|^2 dy=\int_{F} |\hat{\phi}(y)|^2 dy$ holds, 
where $|\cdot|$ is the absolute value on $F$. 

We denote the characteristic function of a subset $A$ of a set $X$ by $\mathrm{ch~}A$. 
In the case $F=\mq_p$, we have the volume vol$(p^m \mz_p)$ of $p^m \mz_p$ equals 
$p^{-m-(\mathrm{ord }\beta/2)}$ and $\widehat{\mathrm{ch~}p^m \mz_p}=
$ vol$(p^m \mz_p)\, \mathrm{ch~} p^{-(m+\mathrm{ord }\beta)} \mz_p$ for $m \in \mz$. 

Put $A_{\phi}=\int_{F} \phi(x)\psi_{\beta}(ax^2)dx$ 
and $B_{\phi}=\int_{F} \hat{\phi}(x)\psi_{\beta}(-x^2/4a)dx$ 
for $a \in F^{\times}$ and $\phi \in S(F)$. 
There exists a constant $\alpha_{\psi_{\beta}}(a) \in \mac$ called the Weil constant 
such that $A_{\phi}=\alpha_{\psi_{\beta}}(a)|2a|^{-1/2}B_{\phi}$ holds. 
It is known that $\alpha_{\psi_{\beta}}(ab^2)=\alpha_{\psi_{\beta}}(a)$ 
for $a, b \in F^{\times}$ 
and that $\alpha_{\psi_{\beta}}(a)=\alpha_{\psi}(a \beta)$, 
where $\psi=\psi_1$. 
Moreover, we have $\alpha_{\psi_{\beta}}(-a)=\overline{\alpha_{\psi_{\beta}}(a)}$ 
and $\alpha_{\psi_{\beta}}(a)^8=1$. 

The Weil representation $\omega_{\psi_{\beta}}$ is a representation of $\wt{F}$ on $S(F)$. 
For $\phi \in S(F)$, we have 
\begin{equation} \label{gexam} \begin{cases} 
\omega_{\psi_{\beta}}(m(a, \tau))\phi(x)=
\tau \alpha_{\psi_{\beta}}(1)\alpha_{\psi_{\beta}}(-a)|a|^{1/2}\phi(ax) \\  
\omega_{\psi_{\beta}}(u^+(b))\phi(x)=\psi_{\beta}(bx^2)\phi(x) \\ 
\omega_{\psi_{\beta}}(N)\phi(x)=|2|^{1/2}\alpha_{\psi_{\beta}}(-1) \hat{\phi}(-2x). \end{cases} \end{equation}
Since $\wt{F}$ is generated by the above elements, 
$\omega_{\psi_{\beta}}$ is determined by these formulas. 
In particular, we have $\omega_{\psi_{\beta}}([g, \tau])\phi=\tau \omega_{\psi_{\beta}}([g])\phi$. 

We define a map $s:$ $\mathrm{SL}_2(\mo) \to \{\pm 1\}$ by 
\begin{equation} \label{split} 
s(g)=\begin{cases} 1 & c \in \mo^{\times} \\
\langle c, d \rangle_F & c \in \map \backslash\{0\} \\ 
\langle -1, d \rangle_F & c=0 \end{cases} \hspace{15pt} 
\text{for } g=\begin{pmatrix} a & b \\ c & d \end{pmatrix} \in \mathrm{SL}_2(\mo). \end{equation} 
If $q$ is odd, we have 
\[
s(g)=\begin{cases} 1 & cd=0 \\
\langle c, d \rangle_F^{\mathrm{ord}\,c} & cd \ne 0. \end{cases}
\]
Recall that the double covering $\wt{F} \to$ $\mathrm{SL}_2(F)$ splits on $\mathrm{SL}_2(\mo)$ 
if and only if $q$ is odd. 
The splitting is given by $g \mapsto [g, s(g)]$. 
Thus if $q$ is odd, a map $\epsilon_F:\wt{\mo} \to \mac^{\times}$ 
defined by $\epsilon_F([g, \tau])=\tau \, s(g)$ is a genuine character. 
If $\mathrm{ord} \psi_{\beta}=0$, we have 
$\omega_{\psi_{\beta}}([g, \tau])\, \mathrm{ch~}\mo$ = 
$\epsilon_F([g, \tau])^{-1} \mathrm{ch~}\mo$. 

\begin{lem} \label{eps} 
Suppose that $q$ is odd. 
Then there exists a genuine character $\epsilon_F: \wt{\mo} \to \mac^{\times}$ such that 
$\omega_{\psi_{\beta}}([g, \tau]) \mathrm{ch~}\mo=\epsilon_F([g, \tau])^{-1} 
\mathrm{ch~}\mo$ if $\mathrm{ord} \psi_{\beta}=0$. 
If $q \ge 5$, then it is a unique genuine character of $\wt{\mo}$. \end{lem} 

\begin{proof} 
We already proved the first part of this lemma. 
If $q \ge 5$, then by (\ref{sub}) we have $U^+(\mo)$, $U^-(\mo) \subset D(\wt{\mo})$. 
Since $N=u^+(-1)u^-(1)u^+(-1)$, $\mathrm{SL}_2(\mo)^{ab}$ is trivial. 
Thus there exists a unique genuine character $\epsilon_F$ of $\wt{\mo}$. \end{proof}

For $a \in \mz$, we define a subset $S(\map^a)$ of $S(F)$ by 
$S(\map^a)=\{f \in S(F) \mid \mathrm{Supp}\,f \subset \map^a \}$. 
For $a \le b \in \mz$, put $S(\map^a/\map^b)=
\{f \in S(\map^a) \mid f(x+t)=f(x)$ for any $t \in \map^b \}$. 
For $f \in S(F) \backslash \{0\}$, 
there exists a pair $(a, b)$ such that $f \in S(\map^a/\map^b)$. 

\begin{lem} \label{gen1} 
Suppose that $q=3$ (resp. $F=\mq_2$). If $\mathrm{ord} \psi_{\beta}=-1$ (resp. $-3$), 
then the group $\wt{\mo}$ preserves $S(\mo/2\map)$ with respect to $\omega_{\psi_{\beta}}$. 
We define $f \in S(\mo/2\map)$ by 
\begin{equation} \label{eigen} 
f(x)=\begin{cases} 1 & \text{ if } x \in 1+2\map \\ -1 & \text{ if } x \in -1+2\map \\ 
0 & \text{ otherwise}. \end{cases} \end{equation} 
Then the subspace of odd functions is $\mac f \subset S(\mo/2\map)$ and 
there exists a genuine character $\mu_{\beta}$ of $\wt{\mo}$ such that 
$\omega_{\psi_{\beta}}([g, \tau])f=\mu_{\beta}([g, \tau])^{-1}f$. 

In the case $q=3$, there exist three genuine characters of $\wt{\mo}$, $\epsilon_F$ and 
$\mu_{\beta}$, 
where $\mu_{\beta}$ extends over all elements $\beta$ such that $\mathrm{ord} \psi_{\beta}=-1$. 
Moreover, the value $\mu_{\beta}(u^+(1))=\psi_{\beta}(-1)$ is a primitive 3rd root of unity, 
which determines $\mu_{\beta}$. \end{lem}

\begin{proof} 
Suppose that $F=\mq_2$ and that $\mathrm{ord} \psi_{\beta}=\mathrm{ord} \beta=-3$. 
It is clear that 
$\wt{\mo}$ preserves $S(\mo/2\map)$ with respect to $\omega_{\psi_{\beta}}$. 
If $\phi \in S(\mo/2\map)$ is an odd function, then $\phi$ satisfies $\phi(x)=\phi(-x)=-\phi(x)$ 
for any $x \in \map$. Thus we have $\phi(\map)=0$. 
Since $F=\mq_2$, we have $\mo^{\times}=(1+2\map) \cup (-1+2\map)$ 
and then $\phi(x)=\phi(1)f \in \mac f$. 

Thus there exists a genuine character $\mu_{\beta}$ of $\wt{\mo}$ such that 
$\omega_{\psi_{\beta}}([g, \tau])f=\mu_{\beta}([g, \tau])^{-1}f$. 
Since $u^-(-1)=Nu^+(1)N^{-1}$ and $N=u^+(-1)u^-(1)u^+(-1)$, 
the value $\mu_{\beta}(u^+(1))$ determines $\mu_{\beta}$. 

Suppose that $q=3$ and that $\mathrm{ord} \psi_{\beta}=-1$. 
Then we prove the first part of the lemma similar to the case above. 
By \cite[\S 2.10]{Ki1}, $\mathrm{SL}_2(\mo)^{ab}$ has order 3. 
Thus there exist three genuine characters of $\wt{\mo}$. 
We have $f(x) \ne 0$ if and only if $x \in \mo^{\times}$. 
By (\ref{gexam}), we have  
$\omega_{\psi_{\beta}}(u^+(b))f(x)=\psi_{\beta}(bx^2)f(x)$ for $b \in \mo$. 
If $f(x) \ne 0$, since we have $x^2-1 \in \map$ by Lemma \ref{order}, 
$\mu_{\beta}(u^+(b))^{-1}=\psi_{\beta}(bx^2)=\psi_{\beta}(b)$. 
In particular, $\mu_{\beta}(u^+(1))=\psi_{\beta}(-1)$ is a primitive 3rd root of unity. 

For $\gamma \in F$ such that $\mathrm{ord} \psi_{\gamma}=-1$, 
if $\mu_{\beta}=\mu_{\gamma}$, then we have $\psi_{\beta}(1)=\psi_{\gamma}(1)$ 
and then $\beta/\gamma \in 1+\map$. 
Since we have $[\mo^{\times}:(1+\map)]=2$, 
the genuine characters of $\wt{\mo}$ are $\epsilon_F$ and $\mu_{\beta}$, 
where $\mu_{\beta}$ extends over all elements $\beta$ such that $\mathrm{ord} \psi_{\beta}=-1$. 
\end{proof}

\begin{lem} \label{gen2} 
In Lemma \ref{gen1}, suppose that $F=\mq_2$ and that $\mathrm{ord} \beta=-3$. 
Then the value $\mu_{\beta}(u^+(1))=e(\beta)$ is a primitive 8th root of unity, 
which determines $\mu_{\beta}$. 
Moreover, there exist four genuine characters $\mu_{\beta}$ of $\wt{\mz_2}$, 
where $\mu_{\beta}$ extends over all elements $\beta$ such that $\mathrm{ord} \beta=-3$. \end{lem}

\begin{proof} By Lemma \ref{gen1}, there exists a genuine character of $\wt{\mz_2}$. 
Since $\mathrm{SL}_2(\mz_2)^{ab}$ has order 4 by \cite[\S 2.10]{Ki1}, 
the number of genuine characters of $\wt{\mz_2}$ is 4. 
We have $\omega_{\psi_{\beta}}(u^+(b))f(x)=e(-\beta bx^2)f(t)$ 
for $f$ in (\ref{eigen}) and $b \in \mz_2$ by (\ref{gexam}). 
If $f(x) \ne 0$, then we have $x^2-1 \in 8\mz_2$ and $e(-\beta bx^2)=e(-\beta b)$. 
In particular, $\mu_{\beta}(u^+(1))=e(\beta)$ is a primitive 8th root of unity. 

For $\gamma \in \mq_2$ such that $\mathrm{ord} \gamma=-3$, 
if $\mu_{\beta}=\mu_{\gamma}$, then we have $e(\beta)=e(\gamma)$ 
and then $\beta/\gamma \in 1+8\mz_2$. 
Since we have $[\mz_2^{\times}:(1+8\mz_2)]=4$, 
there exist four genuine characters $\mu_{\beta}$ of $\wt{\mz_2}$, 
where $\mu_{\beta}$ extends over all elements $\beta$ such that $\mathrm{ord} \beta=-3$. \end{proof} 

Let $\mu_{\beta}$ be a nontrivial genuine character in Lemma \ref{gen1} or Lemma \ref{gen2}. 
Then we have 
\[
\mu_{\beta}([g])=s(g) \kappa(\beta, g) \hspace{15pt} 
g=\begin{pmatrix} a & b \\ c & d \end{pmatrix} \in \mathrm{SL}_2(\mo), 
\]
where $\kappa(\beta, g)$ is a continuous function for $g$. 
In the case $F=\mq_2$ and $\mathrm{ord~}\beta=-3$, we have 
\begin{equation} \label{kappa2}
\kappa(\beta, g)=\begin{cases}\psi_{\beta}(-(a+d)c+3c) & c \in \mz_2^{\times} \\ 
\psi_{\beta}((c-b)d-3(d-1)) & c \in 2\mz_2. \end{cases}
\end{equation}

In the case $q=3$ and $\mathrm{ord} \psi_{\beta}=-1$, we have 
\begin{equation} \label{kappa3}
\kappa(\beta, g)=\psi_{\beta}(-(a+d)c+bd(c^2-1)). 
\end{equation} 

\noindent {\bf Remark.}
Suppose that $K'$ is a compact open subgroup of $\mathrm{SL}_2(\mo)$.
Let $\lambda':\tilde {K'}\rightarrow \mathbb{C}^\times$ be a genuine character.
Then one can show that there exists a continuous function $\kappa'$ on $K'$ such that $\lambda'([g])=s(g)\kappa'(g)$ for any $g\in K'$.
As we do not need this result for the rest of this paper, we omit a proof. \vspace{5pt}


Put $K=$ $\mathrm{SL}_2(\mo)$ and $G=$$\mathrm{SL}_2(F)$. It is known that $K$ (resp. $\tilde{K}$) is 
a compact open subgroup of $G$ (resp. $\tilde{G}$). 
Let $(\pi, V)$ be an irreducible smooth representation of $\tilde{G}$. 
For a character $\lambda$ of $\tilde{K}$, we define a set $(\pi, V)^{\lambda}$ by 
$$(\pi, V)^{\lambda}=\{f \in V \mid \pi(g)f=\lambda(g)^{-1}f \text{ for any } g \in \tilde{K} \}. $$ 

In particular, we consider $(\omega_{\psi_{\beta}}, S(F))^{\lambda}$ 
for a genuine character $\lambda: \tilde{K} \to \mac^{\times}$ 
such that $\lambda(u^+(1)) \ne 1$. 
Since $\lambda(u^+(1)) \ne 1$, $\lambda$ is the character $\mu_{\beta}$ 
in Lemma \ref{gen1} or Lemma \ref{gen2}. 
In particular, we have $q=3$ or $F=\mq_2$. 
When $q=3$ (resp. $F=\mq_2$), we assume that $\mathrm{ord} \psi_{\beta}=-1$ (resp. $-3$). 

\begin{prop} \label{dim}
The representation $c$-Ind$_{\tilde{K}}^{\tilde{G}} \lambda$ is irreducible supercuspidal. 
We have 
$$\mathrm{dim}_{\mac}(\pi, V)^{\lambda}=\begin{cases}1 & 
\pi=c\,\text{-Ind}_{\tilde{K}}^{\tilde{G}} \lambda \\ 0 & \text{otherwise.} \end{cases}$$ 
\end{prop} 

\begin{proof} 
Put $\lambda^g(x)= \lambda(gxg^{-1})$ for $g \in \tilde{G}$. 
We shall prove that 
\begin{equation} \label{hom} \mathrm{Hom}_{g^{-1}Kg \cap K} (\lambda^g, \lambda)=0 
\hspace{15pt} \text{for } g \notin \tilde{K}. 
\end{equation} 
Put $m(a)=\,$diag$(a, a^{-1})$ for $a \in F^{\times}$. 
Since it is known that 
\[
\mathrm{SL}_2(F)=\sum_{n=0}^{\infty} Km(\pi^n)K,
\] 
we have only to consider the case $g=m(\pi^n)$ for $n>0$. 
Since $u^+(1) \in g^{-1}Kg \cap K$, we have $\lambda(u^+(1)) \ne \lambda^g(u^+(1))$, 
which proves (\ref{hom}). 

It is known that (\ref{hom}) implies the first assertion (see \cite[\S11.4]{BuHe1}). 
Although \cite{BuHe1} treated the GL$_2$ case but the proof is valid in our case. 
By \cite[\S2.5]{BuHe1}, we have 
Hom$_{\tilde{G}}(c$-Ind$_{\tilde{K}}^{\tilde{G}} \lambda, \pi) \simeq$ 
Hom$_{\tilde{K}}(\lambda, \pi)$, which completes the proof. 
\end{proof} 

It is known that 
$$\omega_{\psi_{\beta}}=\omega_{\psi_{\beta}}^+ \oplus \omega_{\psi_{\beta}}^-, $$ 
where $\omega_{\psi_{\beta}}^+$ (resp. $\omega_{\psi_{\beta}}^-$) is 
the restriction of $\omega_{\psi_{\beta}}$ to the even (resp. odd) functions. 
Note that these restrictions are irreducible but not isomorphic. 
By Proposition \ref{dim}, we have $\omega_{\psi_{\beta}}^- \simeq$ 
c-Ind$_{\tilde{K}}^{\tilde{G}} \mu_{\beta}$. 
Since $\lambda(u^+(1)) \ne 1$, we have dim$(\omega_{\psi_{\beta}}^+, S(F))^{\lambda}=0$ 
by Proposition \ref{dim}. 
Since we assume that $\mathrm{ord} \psi_{\beta}=-1$ if $q=3$ and that $\mathrm{ord} \beta=-3$ if $F=\mq_2$,  we have 
\begin{equation} \label{dim2} \mathrm{dim}_{\mac}(\omega_{\psi_{\beta}}, S(F))^{\lambda}=
\begin{cases}1 & \lambda=\mu_{\beta} \\ 0 & \text{otherwise.} \end{cases} \end{equation} 

Now we assume that $q$ is odd. 
By Lemma \ref{eps}, there exists a genuine character $\epsilon_F:\wt{\mo} \to \mac^{\times}$ 
and we have $\omega_{\psi_{\beta}}([g, \tau])\, \mathrm{ch~}\mo=\epsilon_F([g, \tau])^{-1} 
\mathrm{ch~}\mo$,  where $\mathrm{ord} \psi_{\beta}=0$. 

\begin{lem} \label{mq5} 
Put $T_{\beta}=(\omega_{\psi_{\beta}}, S(F))^{\epsilon_F}$. 
Then we have 
$$T_{\beta}= \begin{cases} \mac \ \mathrm{ ch~} \map^{-\mathrm{ord} \psi_{\beta}/2} & 
\text{if } \mathrm{ord~}\psi_{\beta} \equiv 0 \text{ mod~2 } \\ 0 & \text{otherwise.}\end{cases} $$ 
\end{lem}

\begin{proof} 
Put $D=\mathrm{ord} \psi_{\beta}$. Suppose that $D=0$. 
Then we have $\epsilon_F(U^+(\mo))=1$ and it is clear that $\wt{\mo}$ preserves 
$S(\mo/\mo)=\mac \mathrm{~ch~}\mo$ with respect to $\omega_{\psi_{\beta}}$. 
Then we have $T_{\beta}=\mac \mathrm{~ch~}\mo=\mac \mathrm{~ch~}\map^{-D/2}$. 
Since we have $\omega_{\psi_{\beta t^2}}(g)=\omega_{\psi_{\beta}}(m(t, 1) g\,m(t, 1)^{-1})$ 
for $t \in F^{\times}$ and $g \in \wt{\mo}$, 
the same is true for $T_{\beta t^2}$ for any $t \in F^{\times}$. 

Assume that $\phi \in T_{\beta} \backslash \{0\}$ with $D=1$. 
Then we have $\omega_{\psi_{\beta}}(h)\phi=\epsilon_F(h)^{-1}\phi$. 
By (\ref{gexam}), 
we have $\psi_{\beta}(bx^2)=1$ for any $b \in \mo^{\times}$ when $\phi(x) \ne 0$. 
In particular, since $\mathrm{ord} b x^2 \ge$ $\mathrm{ord} x^2 \ge -1$, we have $\phi \in S(\mo)$. 

We assume $\phi \in S(\map^a/\map^b)$ such that $a \ge 0$ is maximal and $b$ is minimal. 
A calculation of the Fourier transformation shows that 
$\hat{\phi} \in S(\map^{-1-b}/\map^{-1-a})$. 
Since $\alpha_{\psi_{\beta}}(-1)\hat{\phi}(-2x)=\phi(x)$ by (\ref{gexam}), 
we have $\hat{\phi} \in S(\map^a/ \map^b)$. 
Then $a=-1-b$ is less than 0, which contradicts $a \ge 0$. 
\end{proof}

Set $F^{\times 2}= \{x^2 \mid x \in F^{\times} \}$. 
Assume that $q=3$ or $F=\mq_2$. 
By Lemma \ref{gen1} and Lemma \ref{gen2}, 
there exist genuine characters $\mu_{\beta}$ of $\wt{\mo}$. 

\begin{lem} \label{mq3} 
When $q=3$ (resp. $F=\mq_2$), 
we put $T_{\beta}=(\omega_{\psi_{\beta}}, S(F))^{\mu_{\gamma}}$, 
where $\gamma \in F^{\times}$ such that $\mathrm{ord} \psi_{\gamma}=-1$ (resp. $-3$). 
Then we have 
\begin{equation} \label{f3} \mathrm{dim}~T_{\beta}= \begin{cases} 1 & \text{if } 
\beta/\gamma \in F^{\times 2} \\ 0 & \text{otherwise.} \end{cases} \end{equation} 
In particular, when $\beta/\gamma \in 1+\map$ (resp. $1+8\mz_2$), 
we have $T_{\beta}=\mac f$, where $f$ is the function in (\ref{eigen}). \end{lem} 

\begin{proof} 
We prove the lemma in the case $q=3$. The proof for $F=\mq_2$ is similar. 
Put $D=\mathrm{ord} \psi_{\beta}$. 
Then we may assume that $D \in \{0, -1\}$ in the same way as the proof of Lemma \ref{mq5}. 
By Proposition \ref{dim} and (\ref{dim2}), we have dim~$T_{\beta}=0$ when $D=0$. 
Suppose that $D=-1$ and that $\phi \in T_{\beta}$ is nonzero. 
Then by Lemma \ref{gen1}, we have $\phi \in \mac f$. 
Lemma \ref{gen1} shows that $f$ lies in $T_{\beta}$ if and only if $\beta/\gamma \in 1+\map$. 
We have $1+\map \subset F^{\times 2}$, which completes the proof. \end{proof}

\section{Multiplier systems for $\mathrm{SL}_2(\mo)$} \label{system} 
From now on, let $F$ be a totally real number field such that $[F:\mq]=n$. 
Let $v$ be a place of $F$ and $\ma$ the adele ring of $F$. 
We denote the completion of $F$ at $v$ by $F_v$. 
If $v$ is an infinite place, we write $v \mid \infty$. Otherwise, we write $v < \infty$. 
For $v<\infty$, let $\mo_v$, $\map_v$ and $\md_v$ be the ring of integers of $F_v$, 
the maximal ideal of $\mo_v$ and the different of $F_v/\mq_p$, respectively.

For any $v$, let $\iota_v:F \to F_v$ be the embedding. 
The entrywise embeddings of $\mathrm{SL}_2(F)$ into $\mathrm{SL}_2(F_v)$ 
are also denoted by $\iota_v$. 
Let $\{\infty_1, \cdots, \infty_n\}$ be the set of infinite places of $F$. 
Put $\iota_i=\iota_{\infty_i}$ for $1 \le i \le n$. 
We embed $\mathrm{SL}_2(F)$ into $\mathrm{SL}_2(\mr)^n$ by $r \mapsto (\iota_1(r), \cdots, \iota_n(r))$. 

We define the metaplectic group $\wt{\mr}$ of $\mathrm{SL}_2(\mr)$ similar to 
the case $F_v/\mq_p$. 
Let $\ms$ be a finite set of places of $F$, which contains all places above 2 and $\infty$. 
Set 
$$\mathrm{SL}_2(\ma)_{\ms}=\prod_{v \in \ms} \mathrm{SL}_2(F_v) \times 
\prod_{v \notin \ms} \mathrm{SL}_2(\mo_v). $$ 
The double covering of $\mathrm{SL}_2(\ma)_{\ms}$ defined by the 2-cocycle 
$\prod_{v \in \ms} c_v(g_{1, v}, g_{2, v})$ is denoted by $\wt{\ma}_{\ms}$, 
where $c_v$ is the Kubota 2-cocycle for $\mathrm{SL}_2(F_v)$. 

Let $s_v:$ $\mathrm{SL}_2(\mo_v) \to \{\pm 1\}$ is the map $s$ in (\ref{split}) for 
$v<\infty$. 
For a finite set $\ms'$ of places of $F$ such that $\ms \subset \ms'$, 
we can define an embedding 
$\iota_{\ms}^{\ms'}: \wt{\ma}_{\ms} \to \wt{\ma}_{\ms'}$ by 
$$[(g_v), \zeta] \mapsto \Big[(g_v),\ \ \zeta \prod_{v \in \ms' \backslash \ms} s_v(g_v) \Big]. $$ 

For $v<\infty$, a map $\mathbf{s}_v: \mathrm{SL}_2(\mo_v) \to \wt{\mo_v}$ is given by 
$\mathbf{s}_v(\gamma)=[\gamma, s_v(\gamma)]$ for $\gamma \in \mathrm{SL}_2(\mo_v)$. 
The adelic metaplectic group $\wt{\ma}$ is the direct limit $\varinjlim \wt{\ma}_{\ms}$. 
It is a double covering of $\mathrm{SL}_2(\ma)$ and 
there exists a canonical embedding $\wt{F_v} \to \wt{\ma}$ for each $v$. 
Let $\prod_{v}' \wt{F_v}$ be the restricted direct product with respect to 
$\mathbf{s}_v($$\mathrm{SL}_2(\mo_v))$. 
Then there is a canonical surjection $\prod_{v}' \wt{F_v} \to \wt{\ma}$. 
The image of $(g_v)_v \in \prod_{v}' \wt{F_v}$ is also denoted by $(g_v)_v$. 
Note that for a given $g \in \wt{\ma}$, the expression $g=(g_v)_v$ is not unique. 

We denote the embedding of $\mathrm{SL}_2(F)$ into $\mathrm{SL}_2(\ma)$ by $\iota$. 
The finite part of $\mathrm{SL}_2(\ma)$ is denoted by $\mathrm{SL}_2(\ma_f)$. 
Let $\iota_f:$ $\mathrm{SL}_2(F) \to$ $\mathrm{SL}_2(\ma_f)$ be 
the projection of the finite part 
and $\iota_{\infty}:$ $\mathrm{SL}_2(F) \to$ $\mathrm{SL}_2(F_{\infty})=$ $\mathrm{SL}_2(\mr)^n$ that of the infinite part. 
Then we have $\iota(g)=\iota_f(g) \iota_{\infty}(g)$ for any $g \in \mathrm{SL}_2(F)$. 
The embedding of $F$ into $\ma_f$ is also denoted by $\iota_f$. 

It is known that $\mathrm{SL}_2(F)$ can be canonically embedded into $\wt{\ma}$.
The embedding $\tilde{\iota}$ is given by $g \mapsto ([\iota_v(g)])_v$ 
for each $g \in$ $\mathrm{SL}_2(F)$. 
We define the maps $\tilde{\iota}_f:\mathrm{SL}_2(F) \to \wt{\ma_f}$ 
and $\tilde{\iota}_{\infty}:\mathrm{SL}_2(F) \to \wt{F_{\infty}}$ by 
\[
\tilde{\iota}_f(g)=([\iota_v(g)])_{v<\infty} \times ([1_2])_{v \mid \infty}, \hspace{15pt} 
\tilde{\iota}_{\infty}(g)=([1_2])_{v<\infty} \times ([\iota_i(g)])_{v \mid \infty}. 
\]
Then we have $\tilde{\iota}(g)=\tilde{\iota}_f(g) \tilde{\iota}_{\infty}(g)$ for any 
$g \in \mathrm{SL}_2(F)$. 

For $\gamma=[g, \tau] \in \wt{\mr}$, $g=\begin{pmatrix} a & b \\ c& d \end{pmatrix}$ 
and $z \in \mh$, 
$\tilde{j}:\wt{\mr} \times \mh \to \mac$ is an automorphy factor given by 
\begin{equation} \label{auto} 
\tilde{j}(\gamma, z)=\begin{cases}\tau \sqrt{d} & \text{if } c=0, d>0, \\ 
-\tau \sqrt{d} & \text{if } c=0, d<0, \\ \tau(cz+d)^{1/2} & \text{if } c \ne 0. \end{cases} 
\end{equation} 
Here, we choose arg$(cz+d)$ such that $-\pi<$arg$(cz+d) \le \pi$. 
Note that $\tilde{j}([g, \tau], z)$ is the unique automorphy factor such that 
$\tilde{j}([g, \tau], z)^2=j(g, z)$, where $j(g, z)$ is the usual automorphy factor 
on $\mathrm{SL}_2(\mr) \times \mh$ (see \cite[\S 7]{HiIk1}). 
Note that $\tilde{j}([g], z)=J(g, z)$, where $J(g, z)$ is defined in (\ref{capJ}).

\begin{dfn}\label{mulsys}
Let $\Gamma \subset$ $\mathrm{SL}_2(\mo)$ be a congruence subgroup. 
the map $\mv=\mv(\gamma):\Gamma \to \mac^{\times}$ is 
said to be a multiplier system of half-integral weight 
if $\mv(\gamma) \prod_{i=1}^n \tilde{j}([\iota_i(\gamma)], z_i)$ is an automorphy factor 
for $\Gamma \times \mh^n$, where $\tilde{j}$ is the automorphy factor in (\ref{auto}). 
\end{dfn}

We have 
$\tilde{j}(\gamma_1, \gamma_2(z)) \tilde{j}(\gamma_2, z)=\tilde{j}(\gamma_1\gamma_2, z)$ 
for $\gamma_1, \gamma_2 \in \wt{\mr}$. 
Replacing $\gamma_i$ with $[g_i]$ for $i=1, 2$, we have 
\begin{equation} \label{auto2} 
\tilde{j}([g_1], g_2(z)) \tilde{j}([g_2], z)=c_{\mr}(g_1, g_2) \tilde{j}([g_1g_2], z), \end{equation} 
where $c_{\mr}(\cdot, \cdot)$ is the Kubota 2-cocycle at infinite places. 

\begin{lem} \label{muliff}
A function $\mv:\Gamma \to \mac^{\times}$ is a multiplier system of half-integral weight 
if and only if we have 
$$\mv(\gamma_1)\mv(\gamma_2)=c_{\infty}(\gamma_1, \gamma_2) \mv(\gamma_1\gamma_2) 
\hspace{15pt} \gamma_1, \gamma_2 \in \Gamma, $$ 
where $c_{\infty}(\gamma_1, \gamma_2)=
\prod_{i=1}^n c_{\mr}(\iota_i(\gamma_1), \iota_i(\gamma_2))$. \end{lem} 

\begin{proof} 
We have $\iota_i(\gamma_1)\iota_i(\gamma_2)=\iota_i(\gamma_1\gamma_2)$ for any $i$. 
Thus (\ref{auto2}) and Definition \ref{mulsys} prove the lemma. \end{proof}

Let $K_{\Gamma} \subset$ $\mathrm{SL}_2(\ma_f)$ be the closure of $\iota_f(\Gamma)$ 
in $\mathrm{SL}_2(\ma_f)$. 
Then $K_{\Gamma}$ is a compact open subgroup 
and we have $\iota_f^{-1}(K_{\Gamma})=\Gamma$. 
Let $\tilde{K}_{\Gamma}$ be the inverse image of $K_{\Gamma}$ in $\wt{\ma_f}$. 

\begin{lem} \label{vlambda}
Let $\lambda:\tilde{K}_{\Gamma} \to \mac^{\times}$ be a genuine character. 
Put $\mv_{\lambda}(\gamma)=\lambda(\tilde{\iota}_f(\gamma))$ 
for $\gamma \in \Gamma$. 
Then $\mv_{\lambda}$ is a multiplier system of half-integral weight for $\Gamma$. \end{lem} 

\begin{proof} 
For $\gamma_1, \gamma_2 \in \Gamma$, we have 
$\tilde{\iota}(\gamma_1)\tilde{\iota}(\gamma_2)=\tilde{\iota}(\gamma_1\gamma_2)$. 
The left-hand side equals $\tilde{\iota}_f(\gamma_1)\tilde{\iota}_{\infty}(\gamma_1) 
\tilde{\iota}_f(\gamma_2)\tilde{\iota}_{\infty}(\gamma_2)$. 
Since $\tilde{\iota}_{\infty}(g)=([\iota_i(g)])_{i=1, \cdots, n}$ for $g \in$ $\mathrm{SL}_2(F)$
and $\tilde{\iota}_{\infty}(\gamma_1)$ commutes with $\tilde{\iota}_f(\gamma_2)$, 
we have $$\tilde{\iota}_f(\gamma_1)\tilde{\iota}_f(\gamma_2)
=\tilde{\iota}_f(\gamma_1\gamma_2)[1_2, c_{\infty}(\gamma_1, \gamma_2)]. $$ 
Since $\lambda$ is genuine, Lemma \ref{muliff} proves the lemma.
\end{proof} 

For $v<\infty$, the map $\mathbf{s}_v$ is the splitting on $K_1(4)_v$, where 
$$K_1(4)_v=\{\gamma=\begin{pmatrix} a & b \\ c & d \end{pmatrix} \in \mathrm{SL}_2(\mo_v) 
\mid c \equiv 0, d \equiv 1 \text{ mod~4}\}. $$ 
If $K_{\Gamma} \subset K_1(4)_f=\prod_{v<\infty} K_1(4)_v$, 
we may define a splitting $\mathbf{s}:K_{\Gamma} \to \wt{\ma}$ by 
\[
\mathbf{s}(\gamma)=(\mathbf{s}_v(\iota_v(\gamma)))_{v<\infty} \times ([1_2])_{v \mid \infty}. 
\]
We consider it as a homomorphism. 
Then we have $\tilde{K}_{\Gamma}=\mathbf{s}(K_{\Gamma}) \cdot \{[1_2, \pm1]\}$.
Note that $\mathbf{s}(K_{\Gamma}) \subset \wt{\ma_f}$ is a compact open subgroup. 

For any congruence subgroup $\Gamma$, a map $\mv_0:\Gamma \to \mac^{\times}$ 
is defined by $\mv_0(\gamma)=\prod_{v<\infty} s_v(\iota_v(\gamma))$, 
which is not always a multiplier system of half-integral weight for $\Gamma$. 
\begin{cor} \label{cor1} 
If $\Gamma \subset \Gamma_1(4)=\left\{\begin{pmatrix} a & b \\ c & d \end{pmatrix} \in 
\mathrm{SL}_2(\mo) \mid c \equiv 0, d \equiv 1 \text{ mod~4} \right\}$, 
then $\mv_0$ is a multiplier system of half-integral weight for $\Gamma$. 
\end{cor} 

\begin{proof} 
Since $\Gamma \subset \Gamma_1(4)$, we have $K_{\Gamma} \subset K_1(4)_f$. 
We define a genuine character $\lambda:\tilde{K}_{\Gamma} \to \mac^{\times}$ by 
\[
\lambda(\mathbf{s}(k)[1_2, \tau])=\tau, \hspace{15pt} 
k \in K_{\Gamma}, \tau \in \{\pm1\}. 
\]
Put $\mv_{\lambda}(\gamma)=\lambda(\tilde{\iota}_f(\gamma))$ for $\gamma \in \Gamma$. 
Since $\mathbf{s}(\gamma)=([\iota_v(\gamma), s_v(\iota_v(\gamma))])_{v<\infty}$, 
we have 
\[
\mv_{\lambda}(\gamma)=\lambda(\mathbf{s}(\gamma)[1_2, \mv_0(\gamma)])
=\mv_0(\gamma). 
\] 
Therefore Lemma \ref{vlambda} proves the corollary. 
\end{proof} 

Now suppose that $\Gamma \subset$ $\mathrm{SL}_2(\mo)$ is a congruence subgroup 
and that $\mv:\Gamma \to \mac^{\times}$ is a multiplier system of half-integral weight. 

\begin{lem} \label{weldef} 
There exists a genuine character $\lambda:\tilde{K}_{\Gamma} \to \mac^{\times}$ 
such that $\mv_{\lambda}=\mv$ 
if and only if there exists a congruence subgroup 
$\Gamma' \subset \Gamma \cap \Gamma_1(4)$ 
such that $\mv(\gamma)=\mv_0(\gamma)$ for any $\gamma \in \Gamma'$. 
\end{lem} 

\begin{proof} 
Suppose that there exists a genuine character 
$\lambda:\tilde{K}_{\Gamma} \to \mac^{\times}$ such that $\mv_{\lambda}=\mv$. 
Since Ker~$\lambda$ and $\mathbf{s}(\Gamma_1(4))$ are open in $\wt{\ma_f}$, 
the intersection is also open. We denote its image in $\mathrm{SL}_2(\ma_f)$ by $K'$. 
Then we have $\tilde{K'}=\mathbf{s}(K') \times \{[1_2, \pm1]\}$. 
Put $\Gamma'=\iota_f^{-1}(K')$. 
Then we have 
$\mv(\gamma)=\mv_{\lambda}(\gamma)=\mv_0(\gamma)$ for any $\gamma \in \Gamma'$. 

Conversely, suppose that there exists a congruence subgroup 
$\Gamma' \subset \Gamma \cap \Gamma_1(4)$ 
such that $\mv(\gamma)=\mv_0(\gamma)$ for any $\gamma \in \Gamma'$. 
Then the closure $K_{\Gamma'}$ of $\iota_f(\Gamma')$ in $\mathrm{SL}_2(\ma_f)$ is 
a compact open subgroup. 
Since $\iota_f(\Gamma)$ is dense and $K_{\Gamma'}$ is open in $K_{\Gamma}$, 
we have $K_{\Gamma}=\iota_f(\Gamma) \cdot K_{\Gamma'}$. 
For $k \in \tilde{K}_{\Gamma}$, there exist $\gamma \in \Gamma$, $k' \in K_{\Gamma'}$ and 
$\tau \in \{\pm1\}$ such that $k=\tilde{\iota}_f(\gamma)\mathbf{s}(k')[1_2, \tau]$. 

We assume that $k$ also equals $\tilde{\iota}_f(\gamma_0)\mathbf{s}(k_0')[1_2, \tau_0]$ 
for $\gamma_0 \in \Gamma$, $k_0' \in K_{\Gamma'}$ and $\tau_0 \in \{\pm1\}$. 
Put $\omega=\gamma_0^{-1}\gamma$. 
Then we have $\omega \in \Gamma'$ and 
$$\tilde{\iota}_f(\gamma)=
\tilde{\iota}_f(\gamma_0)\tilde{\iota}_f(\omega)[1_2, c_{\infty}(\gamma_0, \omega)], 
\hspace{15pt} \tilde{\iota}_f(\omega)= \mathbf{s}(\iota_f(\omega))[1_2, \mv_0(\omega)]. $$ 
Then we have $k=\tilde{\iota}_f(\gamma) \mathbf{s}(k')[1_2, \tau]
=\tilde{\iota}_f(\gamma_0)\mathbf{s}(\iota_f(\omega)k')
[1_2, \tau \mv_0(\omega)c_{\infty}(\gamma_0, \omega)]$. 
Thus we have $k_0'=\iota_f(\omega)k'$ 
and $\tau_0=\tau \mv_0(\omega)c_{\infty}(\gamma_0, \omega)$. 
Since $\mv=\mv_0$ in $\Gamma'$ and  
$\mv(\gamma)=\mv(\gamma_0)\mv(\omega)c_{\infty}(\gamma_0, \omega)$ 
by Lemma \ref{muliff}, we have $\mv(\gamma_0)\tau_0=\mv(\gamma)\tau$. 
Then a function $\lambda(k)=\mv(\gamma)\tau$ is well-defined. 

Since $\lambda(k[1_2, \sigma])=\mv(\gamma)\tau\sigma=\sigma \lambda(k)$ 
for $\sigma \in \{\pm 1\}$, $\lambda$ is genuine. 
It suffices to show that $\lambda(k_1k_2)=\lambda(k_1)\lambda(k_2)$ 
for any $k_1, k_2 \in \tilde{K}_{\Gamma}$. 
There exist $\gamma_i \in \Gamma$, $k_i' \in K_{\Gamma'}$ and 
$\tau_i \in \{\pm1\}$ such that $k_i=\tilde{\iota}_f(\gamma_i)\mathbf{s}(k_i')[1_2, \tau_i]$ 
for $i=1, 2$. 
Then we have $\lambda(k_1)\lambda(k_2)=\mv(\gamma_1)\mv(\gamma_2)\tau_1\tau_2$. 
Replacing $K_{\Gamma'}$ with its sufficiently small subgroup, 
we may assume that $\mathbf{s}(K_{\Gamma'})$ is a normal subgroup of 
$\tilde{K}_{\Gamma}$. 
Then we have 
\[
\tilde{\iota}_f(\gamma_2)^{-1}\mathbf{s}(k_1')\tilde{\iota}_f(\gamma_2)
=\mathbf{s}(\iota_f(\gamma_2)^{-1}k_1'\iota_f(\gamma_2)) \in \mathbf{s}(K_{\Gamma'}).
\]
Since $\tilde{\iota}_f(\gamma_1)\tilde{\iota}_f(\gamma_2)
=\tilde{\iota}_f(\gamma_1\gamma_2)[1_2, c_{\infty}(\gamma_1, \gamma_2)]$, 
$\lambda(k_1k_2)$ equals
$\mv(\gamma_1\gamma_2)c_{\infty}(\gamma_1, \gamma_2)\tau_1\tau_2$. 
By Lemma \ref{muliff}, we have $\lambda(k_1k_2)=\lambda(k_1)\lambda(k_2)$, 
which proves the lemma. 
\end{proof} 

\begin{prop} \label{notmq}
If $F \ne \mq$, then any multiplier system $\mv$ of half-integral weight 
of any congruence subgroup $\Gamma \subset \mathrm{SL}_2(\mo)$ is obtained 
from a genuine character of $\tilde{K}_{\Gamma}$. 
\end{prop}

\begin{proof}
By Lemma \ref{weldef}, it suffices to show that there exists a congruence subgroup 
$\Gamma' \subset \Gamma \cap \Gamma_1(4)$ 
such that $\mv(\gamma)=\mv_0(\gamma)$ for any $\gamma \in \Gamma'$. 
We assume that a congruence subgroup $\Gamma$ satisfies 
$\Gamma \subset \Gamma_1(4)$ by replacing $\Gamma$ with $\Gamma \cap \Gamma_1(4)$. 
Since $\mv_0(\gamma)/\mv(\gamma)$ is a character of $\Gamma$, 
we have $\mv_0(\gamma)/\mv(\gamma)=1$ for any $\gamma \in D(\Gamma)$. 
By the congruence subgroup property, $D(\Gamma)$ contains 
a congruence subgroup $\Gamma'$ (see \cite[Corollary 3 of Theorem 2]{Se2} or \cite[\S 3]{Ki1}). 
Thus we have $\mv(\gamma)=\mv_0(\gamma)$ for any $\gamma \in \Gamma'$, 
which proves this proposition. 
\end{proof}

By Lemma \ref{weldef} and Proposition \ref{notmq}, the multiplier system of half-integral 
weight of a congruence subgroup $\Gamma$ associated to an automorphy factor 
in the sense of Shimura \cite{Sh1} is obtained from a genuine character of $\tilde K_\Gamma$.

\begin{lem} \label{mv0} 
If $F=\mq$, then we have 
\[
\mv_0(g)=\begin{cases}\left(\dfrac{d}{c}\right)^* & c:\text{ odd} \vspace{5pt} \\
\left(\dfrac{c}{d}\right)_* & c:\text{ even}, \end{cases} \hspace{15pt} 
g=\begin{pmatrix} a & b \\ c & d \end{pmatrix} \in \mathrm{SL}_2(\mz). 
\]
\end{lem} 

\begin{proof} 
In the case $(c, d)=(\pm1, 0)$, we have $\left(\dfrac{0}{c}\right)^*=\mv_0(g)=1$. 
In the case $(c, d)=(0, 1)$ (resp. $(0, -1)$), 
we have $\left(\dfrac{0}{d}\right)_*=\mv_0(g)=1$(resp. $-1$). 
If $c \ne 0$ and $d \in 2\mz+1$ satisfy $(c, d)=1$, we have 
\[
\left(\frac{c}{d}\right)^*=\left(\frac{c}{|d|}\right), \hspace{15pt} 
\left(\frac{c}{d}\right)_*=t(c, d) \left(\frac{c}{|d|}\right), \hspace{15pt} 
t(c, d)=\begin{cases}-1 & c, d<0 \\ 1 & \text{otherwise. } \end{cases} 
\] 

Suppose that $cd \ne 0$. 
Put $u=c \cdot 2^{-\mathrm{ord}_2c}$. Then we have $(u, d)=(c, d)=1$. 
Put $t_0(x, y)=(-1)^{(x-1)(y-1)/4}$ for $x, y \in 2\mz+1$. 
If a prime $p$ satisfies $p \mid c$, we have 
\[
\langle c, d \rangle_p=\begin{cases}\left(\dfrac{d}{p}\right)^{\mathrm{ord}_pc}&p\geq 3\\ 
t_0(u, d) \left(\dfrac{2}{|d|}\right)^{\mathrm{ord}_2c}&p=2. \end{cases}
\]
If $c$ is odd, then we have 
$\left(\dfrac{d}{c}\right)^*=\prod_{p|c} \left(\dfrac{d}{p}\right)^{\text{ord}_pc}
=\mv_0(g)$. 
If $c$ is even, then we have 
$\left(\dfrac{u}{|d|}\right)\left(\dfrac{d}{|u|}\right)=t(c, d)t_0(u, d)$ 
(see \cite[p.51]{Kn1}). 
Thus we have 
\[
\left(\dfrac{c}{d}\right)_*
=t(c, d) \left(\dfrac{2}{|d|}\right)^{\mathrm{ord}_2c}\left(\dfrac{u}{|d|}\right)
=t_0(u, d) \left(\frac{2}{|d|} \right)^{\mathrm{ord}_2c} \left(\frac{d}{|u|} \right)=\mv_0(g). 
\]
\end{proof}

Put 
\[
K_f=\prod_{v<\infty} \mathrm{SL}_2(\mo_v). 
\]
Then $K_f$ is a compact open group of $\mathrm{SL}_2(\ma_f)$.
The inverse image of $K_f$ in $\wt{\ma_f}$ is denoted by $\tilde{K_f}$. 
We have 
$\mathrm{SL}_2(\mo)=\mathrm{SL}_2(F) \cap K_f \cdot\mathrm{SL}_2(F_{\infty})$.

\begin{prop} \label{lift} 
Let $\mv$ be a multiplier system of half-integral weight for $\mathrm{SL}_2(\mo)$. 
Then there exists a genuine character 
$\lambda:\tilde{K}_f \to \mac^{\times}$ such that $\mv_{\lambda}=\mv$. 
\end{prop}

\begin{proof} 
If $F \ne \mq$, the assertion is proved by Proposition \ref{notmq}. 
If $F=\mq$, let $\mv_{\eta}$ be the multiplier system of $\eta(z)$ in (\ref{etamul}). 
Put 
$$\Gamma(12)=\left\{ \begin{pmatrix}a & b \\ c & d \end{pmatrix} \in \mathrm{SL}_2(\mz) 
\mid a \equiv d \equiv 1, b \equiv c \equiv 0 \text{ mod~12}\right\}. $$ 
By Lemma \ref{mv0}, we have $\mv_{\eta}(\gamma)=\mv_0(\gamma)$ 
for $\gamma \in \Gamma(12)$. 
Since $\mv_{\eta}(\gamma)/\mv(\gamma)=1$ for any $\gamma \in D(\mathrm{SL}_2(\mz))$, 
we have $\mv(\gamma)=\mv_0(\gamma)$ for any 
$\gamma \in D(\mathrm{SL}_2(\mz)) \cap \Gamma(12)$, which is a congruence subgroup. 
By Lemma \ref{weldef}, there exists a genuine character 
$\lambda:\tilde{K}_f \to \mac^{\times}$ such that $\mv_{\lambda}=\mv$.  
\end{proof} 

\begin{cor} \label{cor2} 
There exists a multiplier system $\mv$ of half-integral weight for $\mathrm{SL}_2(\mo)$ 
if and only if 2 splits completely in $F/\mq$. 
There exists a genuine character of $\wt{\mo_v}$ for any $v<\infty$, 
provided that this condition holds. 
\end{cor} 

\begin{prop} \label{kappa}
Suppose that 2 splits completely in $F/\mq$. 
Let $\mv_{\lambda}$ be a multiplier system of half-integral weight of $\mathrm{SL}_2(\mo)$, 
where $\lambda=\prod_{v<\infty} \lambda_v$ is a genuine character of $\tilde{K}_f$. 
Put $S_2=\{v < \infty \mid F=\mq_2 \}$ and $T_3=\{v<\infty \mid q_v=3\}$. 
If $q_v$ is odd, let $\epsilon_v$ be a genuine character of $\mathrm{SL}_2(\mo_v)$ 
in Lemma \ref{eps}. 
We set $S_3 =\{v \in T_3 \mid \lambda_v \ne \epsilon_v\}$. 
Let $\beta_v$ be a element of $F^{\times}$ such that $\lambda_v=\mu_{\beta_v}$ 
for $v \in S_2 \cup S_3$. 
Then we have 
\[
\mv_{\lambda}(\gamma)=\mv_0(\gamma) 
\prod_{v \in S_2 \cup S_3} \kappa_v(\beta_v, \iota_v(\gamma)) \hspace{20pt} 
\gamma= \begin{pmatrix} a & b \\ c & d \end{pmatrix} \in \mathrm{SL}_2(\mo). 
\]
Here, if $v \in S_2$, 
\[
\kappa_v(\beta_v, g)=\begin{cases}\psi_{\beta_v}(-(a+d)c+3c) & c \in \mz_2^{\times} \\ 
\psi_{\beta_v}((c-b)d-3(d-1)) & c \in 2\mz_2 \end{cases} 
\]
and if $v \in S_3$, 
\[
\kappa_v(\beta_v, g)=\psi_{\beta_v}(-(a+d)c+bd(c^2-1)) 
\]
for $g= \begin{pmatrix} a & b \\ c & d \end{pmatrix} \in \mathrm{SL}_2(\mo_v)$. 
Note that $\kappa_v(\beta_v, \iota_v(\gamma))$ is a continuous function on $\gamma$. 
\end{prop}

\begin{proof}
We have $\mv_{\lambda}(\gamma)=\lambda(\tilde{\iota}_f(\gamma))
=\prod_{v<\infty}\lambda_v([\iota_v(\gamma)])$. 
If $v \notin S_2 \cup S_3$, then we have $\epsilon_v([g])=s_v(g)$ 
for any $g \in \mathrm{SL}_2(\mo_v)$.
If $v \in S_2$ (resp. $S_3$), we have 
$\mu_{\beta_v}([g])=s_v(g) \kappa_v(\beta_v, g)$ by (\ref{kappa2}) (resp. (\ref{kappa3})) 
for any $g \in \mathrm{SL}_2(\mo_v)$. 
This proves the proposition. 
\end{proof}

\section{The condition of the existence of a theta function} \label{main} 
Suppose that 2 splits completely in $F/\mq$. 
By Lemma \ref{noexi}, 
there exists a genuine character $\lambda_v:\wt{\mo_v} \to \mac^{\times}$ for any $v<\infty$. 
If $v<\infty$, put $K_v=$ $\mathrm{SL}_2(\mo_v)$. 
If $v \mid \infty$, put $K_v=\mathrm{SO(2)}$. 
Then $K_v$ is a maximal compact subgroup of $\mathrm{SL}_2(F_v)$ for any $v$. 
Let $\beta$ be an element of $F^{\times}$ 
and $\psi_{\beta}$ the character of $\ma/F$ as in Section \ref{intro}. 
For any $v$, we denote the Weil representation of $\wt{F_v}$ by $\omega_{\psi_{\beta}, v}$. 

Let $\{\infty_1, \cdots, \infty_n\}$ be the set of $v \mid \infty$ 
and $S(\mr)$ the Schwartz space of $\mr$.
We have an irreducible decomposition 
\[
\omega_{\psi_{\beta}, v}=\omega_{\psi_{\beta}, v}^+ \oplus \omega_{\psi_{\beta}, v}^-,
\]
where $\omega_{\psi_{\beta}, v}^+$ (resp. $\omega_{\psi_{\beta}, v}^-$) is 
an irreducible representation of the set of even (resp. odd) functions in $S(\mr)$ 
 (see \cite[Lemma 2.4.4]{LiVe1}). 

The group $\wt{\mr}$ has a maximal compact subgroup $\widetilde{\mathrm{SO}(2)}$, 
which is the inverse image of $\mathrm{SO}(2)$ in $\wt{\mr}$. 
It is known that if $\lambda_v: \widetilde{\mathrm{SO(2)}} \to \mac^{\times}$ is 
a genuine character, 
dim$_{\mac}(\omega_{\psi_{\beta}, v}, S(\mr))^{\lambda_v}$ is at most 1. 
Let  $\lambda_{\infty, 1/2}$  be a genuine character of lowest weight 1/2 
with respect to $(\omega_{\psi_{\beta}, v}^+, S(\mr))$ and $\lambda_{\infty, 3/2}$ 
of lowest weight 3/2 with respect to $(\omega_{\psi_{\beta}, v}^-, S(\mr))$. 
For $\beta>0$, 
$(\omega_{\psi_{\beta}, v}^+, S(\mr))^{\lambda_{\infty, 1/2}}= \mac\, e(i\iota_v(\beta) x^2)$ 
and 
$(\omega_{\psi_{\beta}, v}^-, S(\mr))^{\lambda_{\infty, 3/2}}=\mac\, x e(i\iota_v(\beta)x^2)$ 
are spaces of lowest weight vectors. 
If $\beta<0$, there exist no lowest weight vectors with respect to 
$(\omega_{\psi_{\beta}, v}^+, S(\mr))$ or $(\omega_{\psi_{\beta}, v}^-, S(\mr))$. 

Note that $\lambda_v(\mathbf{s}_v($$\mathrm{SL}_2(\mo_v)))=1$ 
for any $v<\infty$ except for finitely many places. 
Then a genuine character $\lambda_f:\tilde{K_f} \to \mac^{\times}$ is given by 
$\lambda_f(g)=\prod_{v<\infty} \lambda_v(g_v)$ for $g=(g_v)_v \in \tilde{K_f}$. 
Put $w=(w_1, \cdots, w_n) \in \{1/2, 3/2 \}^n$. 
We define an automorphy factor $j^{\lambda_f, w}(\gamma, z)$ 
for $\gamma \in$ $\mathrm{SL}_2(\mo)$ and $z=(z_1, \cdots, z_n) \in \mh^n$ by 
$$j^{\lambda_f, w}(\gamma, z)=\prod_{v<\infty} \lambda_v([\iota_v(\gamma)]) 
\prod_{i=1}^n \tilde{j}([\iota_i(\gamma)], z_i)^{2w_i}.$$ 
In particular, we have 
$j^{\lambda_f, w}(-1_2, z)=\prod_{v<\infty} \lambda_v([-1_2]) \times (-1)^{\sum 2w_i}$. 
If it does not equal 1, the space of Hilbert modular forms of weight $w$ for $\mathrm{SL}_2(\mo)$ 
is $\{0\}$. 

Put $K=K_f \times \prod_{v \mid \infty} \mathrm{SO}(2)$. 
There exists a genuine character $\lambda:\tilde{K} \to \mac^{\times}$ 
such that its $v$-component equals $\lambda_v$, 
where $\lambda_{\infty_i}$ is $\lambda_{\infty, 1/2}$ or $\lambda_{\infty, 3/2}$ 
for $1 \le i \le n$. 
Then we have an automorphy factor $j^{\lambda_f, w}(\gamma, z)$ 
corresponding to $\lambda$ such that $\lambda_{\infty_i}=\lambda_{\infty, w_i}$.

Let $M_w($$\mathrm{SL}_2(\mo), \lambda_f)$ be the space of Hilbert modular forms on $\mh^n$ 
with respect to $j^{\lambda_f, w}(\gamma, z)$. 
A holomorphic function $h(z)$ of $\mh^n$ belongs to the space 
$M_w($$\mathrm{SL}_2(\mo), \lambda_f)$ if and only if 
\[
h(\gamma(z))=j^{\lambda_f, w}(\gamma, z)h(z), 
\]
where $\gamma(z)=(\iota_1(\gamma)(z_1), \cdots, \iota_n(\gamma)(z_n))$
for $\gamma \in$ $\mathrm{SL}_2(\mo)$ and $z \in \mh^n$. 
(When $F=\mathbb{Q}$, the usual cusp condition is also required.)

For each $g \in \wt{\ma}$, there exist $\gamma \in$ $\mathrm{SL}_2(F)$, 
$g_{\infty} \in \widetilde{\mathrm{SL}_2(\mr)^n}$ 
and $g_f \in \tilde{K_f}$ such that $g=\gamma g_{\infty} g_f$ 
by the strong approximation theorem for $\mathrm{SL}_2(\ma)$. 
Put $\mathbf{i}=(\sqrt{-1}, \cdots, \sqrt{-1}) \in \mh^n$. 
For $h \in M_w(\mathrm{SL}_2(\mo), \lambda_f)$, put 
$$\varphi_h(g)=h(g_{\infty}(\mathbf{i})) \lambda_f(g_f)^{-1} 
\prod_{i=1}^n \tilde{j}(g_{\infty_i}, \sqrt{-1})^{-2w_i}. $$ 
Then $\varphi_h$ is an automorphic form on $\mathrm{SL}_2(F)\backslash \wt{\ma}$. 

Let $\mathcal{A}_w(\mathrm{SL}_2(F)\backslash \widetilde{\mathrm{SL}_2(\ma)}, \lambda_f)$ be the space of automorphic forms $\varphi$ on $\mathrm{SL}_2(F)\backslash \widetilde{\mathrm{SL}_2(\ma)}$ satisfying the following conditions (1), (2), and (3).
\begin{itemize}
\item[(1)]
$\varphi(gk_\infty)=\varphi(g)\prod_{i=1}^n \tilde j(k_{\infty, i}, \sqrt{-1})^{-2w_i}$ for any $g\in \widetilde{\mathrm{SL}_2(\ma)}$ and $k_\infty=(k_{\infty, 1}, \ldots, k_{\infty, n})\in \widetilde{\mathrm{SO}(2)^n}$.
\item[(2)] 
$\varphi$ is a lowest weight vector with respect to the right translation of $\widetilde{\mathrm{SL}_2(\mr)^n}$.
\item[(3)]
$\varphi(gk)=\lambda_f(k)^{-1}\varphi(g)$ for any $g\in \widetilde{\mathrm{SL}_2(\ma)}$ and  $k\in \tilde K_f$.
\end{itemize}

Then $\Phi:h\mapsto \varphi_h$ gives rise to an isomorphism 
\[
M_w(\mathrm{SL}_2(\mo), \lambda_f)\longrightarrow\hskip -15pt \raise 4pt\hbox{$\sim$}
\hskip 8pt \mathcal{A}_w(\mathrm{SL}_2(F)\backslash \widetilde{\mathrm{SL}_2(\ma)}, \lambda_f). 
\]
For $\varphi\in \mathcal{A}_w(\mathrm{SL}_2(F)\backslash \widetilde{\mathrm{SL}_2(\ma)}, \lambda_f)$, put $h=\Phi^{-1}(\varphi)$.
Then we have
\[
h(z)=\varphi(g_\infty)\prod_{i=1}^n \tilde{j}(g_{\infty_i}, \sqrt{-1})^{2w_i},
\qquad g_\infty\in \widetilde{\mathrm{SL}_2(\mathbb{R})^n}, \ 
g_\infty(\mathbf{i})=z.
\]

Now suppose that $v<\infty$. 
Let $\md$ be the different of $F/\mq$ 
and $q_v$ the order of the residue field $\mo_v/\map_v$. 
When $q_v$ is odd, let $\epsilon_v$ be a genuine character $\epsilon_F$ in Lemma \ref{eps}. 
Put $S_2=\{v \mid F_v=\mq_2\}$, $T_3=\{v<\infty \mid q_v=3\}$ and 
$S_3=\{v \in T_3 \mid \lambda_v \ne \epsilon_v \}$. 
Since 2 splits completely in  $F/\mq$, we have $|S_2|=n$. 
By Lemma \ref{mq5} and Lemma \ref{mq3}, 
$(\omega_{\psi_{\beta}, v}, S(F_v))^{\lambda_v}$ is not 0 if and only if we have 
$$\mathrm{ord}_v \psi_{\beta, v} \equiv \begin{cases} 0 \text{ mod~2} & 
\text{if } \lambda_v=\epsilon_v \\ 1 \text{ mod~2} & \text{otherwise.} 
\end{cases} $$ 

Then, if $(\omega_{\psi_{\beta}, v}, S(F_v))^{\lambda_v} \ne 0$ for any $v<\infty$, 
there exists a fractional ideal $\maa$ such that 
\begin{equation} \label{class} (8\beta) \md \prod_{v \in S_3} \map_v=\maa^2. \end{equation} 
The set of totally positive elements of $F$ is denoted by $F_+^{\times}$. 
Replacing $\beta$ with $\beta \gamma^2$ and $\maa$ with $(\maa \gamma)^2$ 
in (\ref{class}) for $\gamma \in F_+^{\times}$, 
we may assume $\mathrm{ord} _v \maa=0$ for $v \in S_2 \cup S_3$. 
Then we have $\mathrm{ord} _v \psi_{\beta, v}=-1$ (resp. $-3$) for $v \in S_3$ (resp. $S_2$). 

Conversely, suppose that there exists a fractional ideal $\maa$ satisfying (\ref{class}) 
for a subset $S_3 \subset T_3$. 
For $v<\infty$, put
\[
\lambda_v=\begin{cases} 
\epsilon_v & \text{ if } \mathrm{ord}_v \psi_{\beta, v} \equiv 0 \ \ \text{mod~2} \\
\mu_{\beta} & \text{ if } \mathrm{ord}_v \psi_{\beta, v} \equiv 1 \ \ \text{mod~2}, 
\end{cases}
\]
where $\mu_{\beta}$ is a genuine character in Lemma \ref{gen1} or Lemma \ref{gen2}. 
By Lemma \ref{mq5} and Lemma \ref{mq3}, 
we have $(\omega_{\psi_{\beta}, v}, S(F_v))^{\lambda_v} \ne 0$ for any $v<\infty$. 
Let $\lambda:\tilde{K} \to \mac^{\times}$ be a genuine character 
such that its $v$-component equals $\lambda_v$, where 
$\lambda_{\infty_i}=\lambda_{\infty, w_i}$ for $w_i \in \{1/2, 3/2\}$. 
Put $S_{\infty}=\{\infty_i \mid w_i=3/2 \}$. 

From now on, suppose that $\beta \in F_+^{\times}$. 
Let $S(\ma)$ be the Schwartz space of $\ma$ 
and $(\omega_{\psi_{\beta}}, S(\ma))^{\lambda}$ the set of functions 
$\phi=\prod_v \phi_v \in S(\ma)$ such that 
$\phi_v \in (\omega_{\psi_{\beta}, v}, S(F_v))^{\lambda_v}$ for any $v$. 
For $\phi \in S(\ma)$, we define the theta function $\Theta_{\phi}$ by 
\begin{equation} \label{tfun} 
\Theta_{\phi}(g)=\sum_{\xi \in F} \omega_{\psi_{\beta}}(g)\phi(\xi) \hspace{15pt} 
g=(g_v) \in \wt{\ma}, \end{equation} 
where $\omega_{\psi_{\beta}}(g)\phi(\xi)
=\prod_v \omega_{\psi_{\beta}, v}(g_v)\phi_v(\iota_v(\xi))$ 
is essentially a finite product. 
We have $\Theta_{\phi}(gk)=\lambda(k)^{-1}\Theta_{\phi}(g)$ 
for any $g \in \wt{\ma}$ and $k \in \tilde{K}_f$. 
If $\phi \in (\omega_{\psi_{\beta}}, S(\ma))^{\lambda}$, 
then $\Theta_{\phi}$ is a Hilbert modular form of weight $w=(w_1, \cdots, w_n)$. 

It is known that 
$$\omega_{\psi_{\beta}}=\bigoplus_{S} \omega_{\psi_{\beta}, S}, \hspace{15pt} 
\omega_{\psi_{\beta}, S}=\left( \bigotimes_{v \in S} \omega_{\psi_{\beta}, v}^- \right) 
\otimes \left( \bigotimes_{v \notin S} \omega_{\psi_{\beta}, v}^+ \right), $$ 
where $S$ ranges over all finite subsets of places of $F$ (see \cite[\S 3.4]{Ga1}). 
We define a map $\Theta$ from $\omega_{\psi_{\beta}}$ to 
the space of automorphic forms on $\widetilde{\mathrm{SL}_2(\mathbb{A})}$
by $\Theta(\phi)(g)=\Theta_{\phi}(g)$. 
Then it is known that 
\begin{equation} \label{gan} 
\mathrm{Im}(\Theta) \simeq \bigoplus_{|S|:\text{even}} \omega_{\psi_{\beta}, S}, 
\end{equation} 
(see \cite[Proposition 3.1]{Ga1}).

Let $\mathbf{G}$ be the set of triplets $(\beta, S_3, \maa)$ of 
$\beta \in F_+^{\times}$, a subset $S_3 \subset T_3$ and a fractional ideal $\maa$ of $F$ 
satisfying (\ref{class}) and the condition (A), 
\begin{equation}
|S_2|+|S_3|+|S_{\infty}| \in 2\mz. \tag{\text{A}}
\end{equation}
We define an equivalence relation $\sim$ on $\mathbf{G}$ by 
\[
(\beta, S_3, \maa) \sim (\beta', S'_3, \maa') \Longleftrightarrow 
S_3=S'_3, \ \beta'=\gamma^2 \beta, \ \maa'=\gamma \maa 
\text{ for some } \gamma \in F^{\times}. 
\]

\begin{thm} \label{main1} 
Suppose that 2 splits completely in $F/\mq$. 
Let $\beta \in F^\times_+$, $\lambda:\tilde{K} \to \mac^{\times}$ 
and $w_1, \ldots, w_n\in\{1/2, 3/2\}$ be as above.
Then there exists $\phi=\prod_v \phi_v \in (\omega_{\psi_{\beta}}, S(\ma))^{\lambda}$ 
such that $\Theta_\phi\neq 0$ if and only if 
there exists a fractional ideal $\maa$ of $F$ such that $(\beta, S_3, \maa) \in \mathbf{G}$. 
\end{thm} 

\begin{proof} 
Let $\lambda_v:\widetilde{\mathrm{SL}_2(\mathfrak{o}_v)}\rightarrow \mathbb{C}^\times$ 
be the $v$-component of $\lambda$ for any $v<\infty$. 
We already proved that there exists $\prod_{v<\infty} \phi_v \ne 0$ such that 
$\phi_v \in (\omega_{\psi_{\beta}, v}, S(F_v))^{\lambda_v}$ for any $v<\infty$ if and only if 
there exists a fractional ideal $\maa$ of $F$ satisfying (\ref{class}).
Suppose that the equivalent conditions hold. 
Since we have 
$(\omega_{\psi_{\beta}, v}^+, S(\mr))^{\lambda_{\infty, 1/2}}= \mac\, e(i\iota_v(\beta) x^2)$ 
and 
$(\omega_{\psi_{\beta}, v}^-, S(\mr))^{\lambda_{\infty, 3/2}}=\mac\, x e(i\iota_v(\beta) x^2)$
for any $v \mid \infty$, 
there exists a nonzero $\phi=\prod_v \phi_v \in (\omega_{\psi_{\beta}}, S(\ma))^{\lambda}$. 
It is clear that if there exists a nonzero 
$\phi=\prod_v \phi_v \in (\omega_{\psi_{\beta}}, S(\ma))^{\lambda}$, 
$\prod_{v<\infty} \phi_v \ne 0$ satisfies 
$\phi_v \in (\omega_{\psi_{\beta}, v}, S(F_v))^{\lambda_v}$ for any $v<\infty$. 

Suppose there exists a nonzero 
$\phi=\prod_v \phi_v \in (\omega_{\psi_{\beta}}, S(\ma))^{\lambda}$. 
Note that
$|S_2|+|S_3|+|S_{\infty}|$ is the number of $v$ such that $\phi_v$ is an odd function. 
Then $|S|$ in (\ref{gan}) is $|S_2|+|S_3|+|S_{\infty}|$. 
By (\ref{gan}), it is clear that $\Theta_{\phi} \ne 0$ if and only if the condition (A) holds. \end{proof} 

Let $H$ be a group of fractional ideals that consists of all elements of the form
\[
\prod_{\begin{smallmatrix} {v\in T_3}\end{smallmatrix}} \map_v^{e_v}, \hspace{15pt} 
\sum_v e_v \in 2\mz.
\]
Let $\mathrm{Cl}^+$ be the narrow ideal class group of $F$.
Put $\mathrm{Cl}^{+2}=\{\mathfrak{c}^2 \mid \mathfrak{c} \in \mathrm{Cl}^+ \}$. 
We denote the image of the group $H$ (resp. $\mathfrak{b} \in \mathrm{Cl}^+$) 
in $\mathrm{Cl}^+/\mathrm{Cl}^{+2}$ by $\bar{H}$ (resp. $[\mathfrak{b}]$). 

\begin{thm} \label{main2}
Suppose that 2 splits completely in $F/\mq$. 
Let $w_1, \ldots, w_n\in\{1/2, 3/2\}$ be as above.
\begin{itemize}
\item[(1)] 
Suppose that $|S_2|+|S_{\infty}|$ is even.
Then there exists $(\beta, S_3, \maa) \in \mathbf{G}$ if and only if $[\md]\in\bar{H}$. 
\item[(2)] 
Suppose that $|S_2|+|S_{\infty}|$ is odd.
Then there exists $(\beta, S_3, \maa) \in \mathbf{G}$ if and only if $T_3\neq \emptyset$ 
and $[\md \map_{v_0}]\in \bar{H}$. Here, $v_0$ is any fixed element of $T_3$.
\end{itemize} 
\end{thm} 

\begin{proof}
We prove the theorem in case (1). The proof for case (2) is similar. 

If $[\md] \in \bar{H}$, we have 
$(8\beta)\md \prod_{\begin{smallmatrix} {v \in T_3} \end{smallmatrix}} 
\map_v^{e_v}= \maa'^2$ 
such that $\sum_v e_v$ is even for a fractional ideal $\maa'$ and $\beta \in F_+^{\times}$. 
Put $S_3=\{v \in T_3 \mid e_v$ : odd$\}$. 
Since $|S_2|+|S_3|+|S_{\infty}|$ is even, we have $(\beta, S_3, \maa) \in \mathbf{G}$, where 
\[
\maa=\prod_{\begin{smallmatrix} {v \in T_3\backslash S_3} \end{smallmatrix}} 
\map_v^{-e_v/2} \maa'. 
\]

Conversely, if there exists $(\beta, S_3, \maa) \in \mathbf{G}$, 
it satisfies (\ref{class}) and $|S_3|$ is even. 
Then we have $[\md]=\prod_{v \in S_3} [\map_v] \in \bar{H}$. 
\end{proof}

Let $w_i$ be 1/2 or 3/2 for $1 \le i \le n$. 
Suppose that there exists  $(\beta, S_3, \maa) \in \mathbf{G}$. 
Replacing $(\beta, S_3, \maa)$ with an equivalent element of $\mathbf{G}$, 
we may assume $\mathrm{ord} _v \maa=0$ for $v \in S_2 \cup S_3$. 
Let $f_v$ be the function $f$ in (\ref{eigen}) and put 
\[
f=\prod_{v \in S_2 \cup S_3} f_v \times 
\prod_{v<\infty, v \notin S_2 \cup S_3} \mathrm{ch }\maa_v^{-1}, 
\]
where $\maa_v=\maa \mo_v$. 
Put $\phi=f \times \prod_{i=1}^n f_{\infty, i}$, 
where $f_{\infty, i}(x)=x^{w_i-(1/2)} e(i\iota_i(\beta) x^2)$ for $x \in \mr$. 
By Theorem \ref{main1}, there exists $\Theta_{\phi} \ne 0$ of weight $w=(w_1, \cdots, w_n)$. 

Put $z=(z_1, \cdots, z_n), \mathbf{i}=(\sqrt{-1}, \cdots, \sqrt{-1}) \in \mh^n$. 
We define $x_i, y_i \in \mr$ by $z_i=x_i+\sqrt{-1} y_i$ for $1 \le i \le n$. 
Then we have $z=g_{\infty}(\mathbf{i})$, where 
$g_{\infty}=(g_{\infty_1}, \cdots, g_{\infty_n}) \in$ $\mathrm{SL}_2(\mr)^n$, 
$g_{\infty_i}=\begin{pmatrix} y_i^{1/2} & y_i^{1/2}x_i \\ 0 & y_i^{-1/2} \end{pmatrix}$. 
Since $\lambda_v([1_2])=1$ for $v<\infty$, 
we have $$\Theta_{\phi}(g_{\infty})=\sum_{\xi \in \maa^{-1}} f(\iota_f(\xi)) 
\prod_{i=1}^n \omega_{\psi_{\beta}, \infty_i}([g_{\infty_i}]) f_{\infty, i}(\iota_i(\xi)). $$ 

\begin{thm} \label{hil} 
Let $\phi$ and $\Theta_{\phi}$ be as above. 
We define a theta function $\theta_{\phi}: \mh^n \to \mac$ by 
\[
\theta_{\phi}(z)=\sum_{\xi \in \maa^{-1}} f(\iota_f(\xi)) 
\prod_{\infty_i \in S_{\infty}} \iota_i(\xi) \prod_{i=1}^n e(z_i \iota_i(\beta \xi^2)). 
\] 
Then $\theta_{\phi}$ is a nonzero Hilbert modular form of weight $w$ 
for $\mathrm{SL}_2(\mo)$ with respect to a multiplier system. 

Every theta function of weight $w$ for $\mathrm{SL}_2(\mo)$ with a multiplier system may be obtained in this way. 
\end{thm} 

\begin{proof} 
Since 
\[
\omega_{\psi_{\beta}, \infty_i}([g_{\infty, i}]) f_{\infty, i}(\iota_i(\xi))= 
y_i^{w_i/2} \iota_i(\xi)^{w_i-(1/2)} e(z_i \iota_i(\beta \xi^2)), 
\] 
we have $\theta_{\phi}(z)=\Theta_{\phi}(g_{\infty}) \times \prod_{i=1}^n y_i^{-w_i/2}$. 
Then $\theta_{\phi}$ is nonzero. 
Note that 
\[
\tilde{j}([g_{\infty_i}], \sqrt{-1})^{2w_i}=y_i^{-w_i/2}. 
\]

Since $\phi \in (\omega_{\psi_{\beta}}, S(\ma))^{\lambda}$, we have 
$\Theta_\phi\in \mathcal{A}_w(\mathrm{SL}_2(F)\backslash \widetilde{\mathrm{SL}_2(\ma)}, \lambda_f)$. Then we have 
$\theta_\phi=\Phi^{-1}(\Theta_\phi)\in M_w($$\mathrm{SL}_2(\mo), \lambda_f)$. 
The multiplier system of $\theta_{\phi}$ is $\mv_{\lambda}$ given by 
\[
\mv_{\lambda}(\gamma)
=\mv_0(\gamma) \prod_{v \in S_2 \cup S_3} \kappa_v(\beta, \gamma) \hspace{20pt} 
\gamma \in \mathrm{SL}_2(\mo), 
\]
where $\kappa_v$ for  $v \in S_2 \cup S_3$ is the function in Proposition \ref{kappa}. 

By Proposition \ref{lift}, if $\theta$ is a theta function of weight $w$ 
for $\mathrm{SL}_2(\mo)$ with a multiplier system $\mv$, 
we have a genuine character $\lambda_f$ of $\tilde{K}_f$ such that $\mv=\mv_{\lambda_f}$. 
Let $\lambda=\lambda_f \times \prod_{i=1}^n \lambda_{\infty, w_i}$ 
be a genuine character of $\tilde{K}$. 
Then there exists nonzero $\phi \in (\omega_{\psi_{\beta}}, S(\ma))^{\lambda}$ such that 
$\theta=\theta_{\phi}$ up to constant, 
which completes the proof. 
\end{proof} 

\begin{prop} 
Let $\mathrm{Cl}$ be the usual ideal class group of $F$.
Let $Sq:\mathrm{Cl} \to \mathrm{Cl}^+$ be the homomorphism given by 
$[\maa] \mapsto [\maa^2]$ for a fractional ideal $\maa$ of $F$. 
The number of equivalence classes of $\mathbf{G}$ is equal to 
\[
[E^+:E^2] \sum_{\underset{(\text{A})}{S_3 \subset T_3}} 
|Sq^{-1}([\md \prod_{v \in S_3} \map_v])|, 
\]
where $S_3$ ranges over all subset of $T_3$ satisfying (A). 
Here, $E^+$ is the group of totally positive units of $F$ and 
$E^2$ is the subgroup of squares of units of $F$.   
\end{prop} 

\begin{proof} 
We follow the argument of Hammond \cite{Ha1} Theorem 2.9. 
For given $S_3$ satisfying (A), the number of ideal classes $[\maa]$ such that 
$\maa^2$ is narrowly equivalent to $\md \prod_{v \in S_3} \map_v$ is equal to 
$|Sq^{-1}([\md \prod_{v \in S_3} \map_v])|$. 
Then for a given fractional ideal $\maa$ such that 
$\maa^2$ is narrowly equivalent to $\md \prod_{v \in S_3} \map_v$, 
the number of equivalence classes of triplets of the form $(\beta, S_3, \maa)$ 
such that $\beta \in F_+^{\times}$ satisfying (\ref{class}) is equal to $[E^+:E^2]$. 
\end{proof} 

\section{The case $F=\mq$ or $F$ is a real quadratic field} \label{quad}
Suppose that $F=\mq$. 
If $S_{\infty}=\emptyset$, the equivalence class of $\mathbf{G}$ is 
$\{(1/24, \{3\}, \mz)\}$. 
The theta function obtained by $\{(1/24, \{3\}, \mz)\}$ equals $2\eta(z)$. 
Then its multiplier system equals $\mv_{\eta}$ in (\ref{etamul}). 
If $S_{\infty}=\{\infty\}$, then the equivalence class of $\mathbf{G}$ is 
$\{(1/8, \emptyset, \mz)\}$. 
The theta function obtained by $\{(1/8, \emptyset, \mz)\}$ equals $2\eta^3(z)$. 
Then its multiplier system equals the cubic power of $\mv_{\eta}$.

Now suppose that $F=\mq(\sqrt{D})$, where $D>1$ is a square-free integer. 
When there exists $(\beta, S_3, \maa) \in \mathbf{G}$, one of the followings holds. 
\begin{itemize} 
\item[(C1)] $(8\beta) \md =\maa^2$ and $S_3=\emptyset$. 
\item[(C2)] $(8\beta) \md \map=\maa^2$ such that $N_{F/\mq}(\map)=3$ 
and $S_3=\{\map\}$. 
\item[(C3)] $(8\beta) \md \map \bar{\map}=\maa^2$ 
such that $N_{F/\mq}(\map)=N_{F/\mq}(\bar{\map})=3$ and $S_3=\{\map, \bar{\map}\}$. 
\end{itemize}  
If $|S_{\infty}|$ is even, (C1) or (C3) holds. 
If $|S_{\infty}|$ is odd, (C2) holds.  

Suppose that $D \equiv 1$ mod~8. 
Then 2 splits in $F/\mq$ and we have $\md=(\sqrt{D})$. 

\begin{lem} \label{norm}
Let $N$ be a positive square-free integer. 
Put $L=\mathbb{Q}(\sqrt{-1})$ \ (resp.~$L=\mathbb{Q}(\sqrt{-3})$).
Then the following statements are equivalent.
\begin{itemize}
\item[(a)] $N$ is a norm of an element of $L^\times$.
\item[(b)] $N$ is a norm of an integer of $L$.
\item[(c)] No prime factor of $N$ are inert in $L/\mathbb{Q}$.
\item[(d)] There exist integers $u$ and $v$ such that $N=u^2+v^2$ \ (resp.~$N=3u^2+v^2$).
\end{itemize}
\end{lem}

\begin{proof} 
The statements (b), (c) and (d) are equivalent by \cite[\S 68 and \S 70]{Di1}. 
If $L=\mathbb{Q}(\sqrt{-1})$, although \cite[\S 68]{Di1} treated the case $N$ is odd 
but the proof is valid in general case. 
If $L=\mathbb{Q}(\sqrt{-3})$, \cite[\S 70]{Di1} treated the case $N$ is odd 
and not divisible by 3, but the proof is valid in general case. 
If (b) holds, then clearly (a) holds. 

It suffices to show that if (a) holds, then (c) holds. 
Suppose that $\alpha_N \in L^{\times}$ satisfies $N=N_{L/\mq}(\alpha_N)$. 
If a prime $p$ is inert in $L/\mq$, it is a prime element of $L$ and 
we have $N_{L/\mq}(p)=p^2$. 
Then if $p \mid N$, we have $p \mid \alpha_N$ and $p^2 \mid N$, 
which contradicts that $N$ is square-free. 
\end{proof}

We consider an analogy of the following: 
if $K$ is a real quadratic field, then a necessary and sufficient condition 
that the narrow ideal class of the different of $K$ be a square 
is that the discriminant $D$ of $K$ be the sum of two integer squares
(see \cite{Ha1} Proposition 3.1). 

\begin{lem} \label{squ} 
A necessary and sufficient condition that the narrow ideal class of $\md \map$ is a square 
for a prime ideal $\map$ which has norm 3 
is that $D$ is of the form $3u^2+v^2$ for some $u, v \in \mathbb{N}$. 
\end{lem} 

\begin{proof} 
Suppose that the narrow ideal class of $\md \map$ is a square with $N_{F/\mq}(\map)=3$. 
Then there exists $\sigma \in F_+^{\times}$ and a fractional ideal $\maa$ of $F$ 
such that $(\sigma) \md \map=\maa^2$. 
Taking the norm of both sides, we have $3N_{F/\mq}(\sigma)D=A^2$, where $A$ is the norm of $\maa$. 
Put $\sigma=s+t\sqrt{D}$ for $s, t \in \mathbb{Q}$ such that $s>0$. 
Since $3(s^2-t^2 D)D=A^2$, we have 
\[
D=\left(\frac{tD}{s}\right)^2 + 3\left(\frac{A}{3s}\right)^2. 
\]
Put $L=\mq(\sqrt{-3})$.
Then we have $D\in N_{L/\mq}(L^\times)$. 
Lemma \ref{norm} implies that $D=3u^2+v^2$ for some $u, v \in \mathbb{N}$.

We assume that there exists $u, v \in \mathbb{N}$ such that $D=3u^2+v^2$. 
Since $D \equiv 1$ mod~8, we have $u'=u/2 \in \mz$. 
Put $\rho=(v+\sqrt{D})/2$. 
Then we have $N_{F/\mq}(\rho)=(v^2-D)/4=-3u'^2$. 
Let $\mathfrak{q}=\mathfrak{q}_Q$ be a prime ideal which divides $\rho$, 
where $Q$ is a rational prime which is divisible by $\mathfrak{q}$. 
Since $\rho/Q\notin \mo$, 
we have $N_{F/\mq}(\mathfrak{q})=Q$ and
$\mathrm{ord} _{\mathfrak{q}} \rho=$$\mathrm{ord} _Q 3u'^2$. 
Therefore if $Q \ne 3$, $\mathrm{ord} _{\mathfrak{q}} \rho$ is even. 

Since $N_{F/\mq}(\rho)=-3{u'}^2$, there exists a prime ideal $\mathfrak{q}_3$ of $F$ 
which divides both $3$ and $\rho$. 
Since 3 splits or ramifies in $F/\mq$, $\mathrm{ord} _{\mathfrak{q}_3} \rho$ is odd. Put 
\begin{equation} \label{ham}
\maa=\prod_{\mathfrak{q} \nmid 3} \mathfrak{q}^{(\mathrm{ord}_{\mathfrak{q}} \rho)/2} 
\times \mathfrak{q}_3 ^{(\mathrm{ord}_{\mathfrak{q}_3} \rho +1)/2}. 
\end{equation}
Then we have $(\sqrt{D}\rho)\mathfrak{q}_3= \md \maa^2$. 
Since $\sqrt{D}\rho \in F_+^{\times}$, we have $\md\mathfrak{q}_3=(\md\maa)^2$ 
in $\mathrm{Cl}^+$. 
\end{proof} 

\begin{prop} 
Suppose that $F=\mq(\sqrt{D})$, where $D>1$ is a square-free integer such that 
$D \equiv 1$ mod~8. 
\begin{itemize} 
\item[(1)] There exist $\beta \in F_+^{\times}$ and a fractional ideal $\maa$ satisfying (\text{C1}) 
if and only if $p \equiv 1$ mod~4 for any prime $p \mid D$. 
\item[(2)] There exist $\beta \in F_+^{\times}$ and a fractional ideal $\maa$ satisfying (\text{C2})  
if and only if $p \equiv 0 \text{ or }1$ mod~3 for any prime $p \mid D$. 
\item[(3)] There exists $\beta \in F_+^{\times}$ and a fractional ideal $\maa$ satisfying (\text{C3}) 
if and only if $D \equiv 1$ mod~24 and $p \equiv 1 \text{ mod~4}$ for any prime $p \mid D$. 
\end{itemize}  
\end{prop} 

\begin{proof} 
For a prime ideal $\map$ such that $N_{F/\mq}(\map)=3$, 
the equation $(8\beta) \md \map=\maa^2$ implies that 
the narrow ideal class of $\md \map$ is a square. 
Note that a positive integer $x$ is of the form $3u^2+v^2$ for some $u, v \in \mathbb{N}$ 
if and only if any prime $p$ which divides $x$ satisfies $p \equiv 0$ or $1$ mod~3. 
Then Lemma \ref{squ} proves the second assertion. 

The equation $(8\beta) \md=\maa^2$ implies that 
the narrow ideal class of $\md$ is a square. 
Note that a positive integer $x$ is of the form $u^2+v^2$ for some $u, v \in \mathbb{N}$ 
if and only if any prime $p$ which divides $x$ satisfies $p \equiv 1$ mod~4. 
Then \cite{Ha1} Proposition 3.1 proves the first assertion. 

There exist two distinct prime ideal $\map$ and $\bar{\map}$ such that 
such that $N_{F/\mq}(\map)=N_{F/\mq}(\bar{\map})=3$ if and only if 3 splits in $F/\mq$. 
This condition holds if and only if $D \equiv 1$ mod~24. 
In the case $D \equiv 1$ mod~24, we have $\map \bar{\map}=(3)$. 
Then the equation $(8\beta) \md \map \bar{\map}=\maa^2$ implies that 
the narrow ideal class of $\md$ is a square. 
Thus, similarly to the first assertion, \cite{Ha1} Proposition 3.1 proves the third assertion. 
\end{proof} 

Example: put $D=793=13 \cdot 61$. 
Then there exist $\beta \in F_+^{\times}$ and a fractional ideal $\maa$ satisfying 
any condition of (C1), (C2) or (C3). 
Moreover, $\mathrm{Cl}^+$ has order 8 
and the fundamental unit $\varepsilon$ of $F$ has norm 1. 
For example, put $\rho=(5+\sqrt{D})/2$. 
Since $N_{F/\mq}(\rho)=-3 \cdot 8^2$, we have $(\rho)=\mathfrak{q}_2^6 \mathfrak{q}_3$, 
where $\mathfrak{q}_3=(3, 1-\sqrt{D})$ and $\mathfrak{q}_2=(2, (1+\sqrt{D})/2)$ 
are prime ideals. 
Put $\beta=\rho \sqrt{D}/8$ and $\maa=\md \mathfrak{q}_2^3 \mathfrak{q}_3$. 
Then we have $(8\beta) \md \mathfrak{q}_3=\maa^2$.

\end{document}